\documentclass[11pt,a4paper]{article}

\usepackage{amsmath,amsthm,amsopn,amsfonts,amssymb,dsfont,color}
\usepackage[dvips]{graphicx}
\usepackage{psfrag}
\usepackage{comment}

\DeclareMathAlphabet{\mathbbold}{U}{bbold}{m}{n}

\usepackage{hyperref}

\usepackage{a4,a4wide}
\newtheorem{theo}{\textbf{Theorem}}[section]

\newtheorem{thm}[theo]{\textbf{Theorem}}
\newtheorem{lem}[theo]{\textbf{Lemma}}
\newtheorem{prop}[theo]{\textbf{Proposition}}

\theoremstyle{definition}

\newtheorem{defn}[theo]{\textbf{Definition}}

\theoremstyle{remark}
\newtheorem{rem}[theo]{\textbf{Remark}}

\newcommand\vectorize[1]{\left(\begin{matrix}#1\end{matrix}\right)}

\newcommand{\bsl}{\backslash}

\newcommand\pulsenorm[1]{\underset{x-ct\in(-a_0, a_0)}{\sup}\left(#1\right)}

\def\minmu{\mu^0}
\def\maxmu{\mu^\infty}
\def\mingamma{\gamma^0}
\def\maxgamma{\gamma^\infty}
\def\minr{r^0}
\def\maxr{r^\infty}
\def \l {\lambda _1}
\def \lep {\lambda _1^\ep}

\def\ep{{\varepsilon}}

\def\Lp#1{\mathbf L^{#1}}
\def\Lpper#1{\Lp{#1}_{per}}
\def\Hq#1{\mathbf H^{#1}}
\def\Hqper#1{\Hq{#1}_{per}}
\def\Cont{\mathbf C^0}
\def\Contper{\Cont_{per}}
\def\Lip{\mathbf C^{0,1}}
\def\Lipper{\Lip_{per}}
\def\Lippertheta{\mathbf C^{0,\theta}_{per}}
\def\Cun{\mathbf C^1}
\def\Cunper{\Cun_{per}}

\def\R{\mathbb R}

\newcommand\Homotopyspace{\Cunper(\Omega)}
\newcommand\Normalizationleft[1]{\underset{\Omega_0}{\sup}~(#1)}   

\newcommand\thereiam{} 

\title{Pulsating fronts for  Fisher-KPP systems with mutations
as  models in evolutionary epidemiology}

\date{}
\begin{document}

\maketitle


\begin{center}
{\large\bf Matthieu Alfaro and Quentin Griette \footnote{ IMAG, Universit\'e de
Montpellier, CC051, Place Eug\`ene Bataillon, 34095 Montpellier
Cedex 5, France. E-mail: matthieu.alfaro@umontpellier.fr, quentin.griette@umontpellier.fr}.}\\
[2ex]
\end{center}

\vspace{10pt}

\tableofcontents


\begin{abstract} We consider a periodic reaction diffusion system which, because of competition between $u$ and $v$, does not enjoy the comparison principle. It also takes into account mutations, allowing $u$ to switch to $v$ and vice versa. Such a system serves as a model in evolutionary epidemiology where two types of pathogens compete in a heterogeneous environment while mutations can occur, thus allowing coexistence.

We first discuss the existence of nontrivial positive steady states, using some bifurcation technics. Then, to sustain the possibility of invasion when nontrivial steady states exist, we construct pulsating fronts. As far as we know, this is the first such construction in a situation where comparison arguments are not available.
\\

\\

\end{abstract}


\section{Introduction}\label{s:intro}

This work is concerned with the heterogeneous reaction diffusion system
\begin{equation}\label{syst}
\begin{cases}
\partial _t u=\partial _{xx}u+u\left[r_u(x)-\gamma _u(x)(u+v)\right]+\mu (x)v-\mu (x)u, \quad t>0, \, x\in \R,\vspace{5pt}
\\
\partial _t v=\partial _{xx} v+v\left[r_v(x)-\gamma _v(x)(u+v)\right]+\mu (x)u-\mu (x)v, \quad\,\, t>0, \, x\in \R,
\end{cases}
\end{equation}
where $r_u$, $r_v$ are periodic functions and $\gamma _u$, $\gamma _v$, $ \mu $ are periodic positive functions. After discussing the existence of nontrivial steady states via bifurcation technics, we construct pulsating fronts, despite the lack of comparison principle for \eqref{syst}. Before going into mathematical details, let us  describe the relevance of the the above system in evolutionary epidemiology.

\medskip

System \eqref{syst} describes a theoretical population divided into two genotypes with respective densities $ u(t, x) $ and $ v(t, x) $, and living in  a 
one-dimensional habitat $ x\in\mathbb R $. We assume that each genotype yields a different phenotype which also undergoes the influence of the environment. The 
difference in phenotype is expressed in terms of growth rate, mortality and competition, but we assume that the diffusion of the individuals is the same for each 
genotype. Finally, we take into account mutations occuring between the two genotypes. 

The reaction coefficients $ r_u $ and $ r_v $ represent the intrinsic growth rates, which depend on the environment and take into account both birth and death rates. Notice that
$r_u$ and $r_v$ may take some negative values, in deleterious areas where the death rate is greater than the birth rate. Function $\mu $ corresponds to the mutation rate between the two species. It imposes a truly \textit{cooperative} dynamics in the small populations regime, and couples 
the dynamics of the two species. In particular, one expects that, at least for small mutation rates, \textit{mutation aids survival and coexistence}. We also make the 
assumption that the mutation process is symmetric. From the mathematical point of view, this
simplifies some of the arguments we use and improves the readability of the paper. We have no doubt that similar results hold in the non-symmetric case, though the proofs may be more involved.

In this context, the ability of the
species to survive \textit{globally in space}  depends on the sign of the principal eigenvalue of the linearized operator around extinction $(0,0)$, as we will show further, which involves the coefficients $r_u$, $r_v$, $\mu$.

Finally, $ \gamma_u $ and $ \gamma_v $ represent the strength of the competition (for e.g. a finite resource) between the two strains. The associated dynamics arises when populations begin to grow.  It has no influence on the survival of the two species, but 
regulates 
the equilibrium densities of the two populations. 

Such a framework is particularly suited to model the propagation of a pathogenic species within a population of hosts. Indeed system \eqref{syst} can easily be 
derived from a host-pathogen microscopic model \cite{Gri-Rao-Gan-15} in which we neglect the influence of the pathogen on the host's diffusion. 

\medskip

In a homogeneous environment the role of mutations, allowing survival for both $u$ and $v$,  has recently been
studied by Griette and Raoul \cite{Gri-Rao}, through the  system
\begin{equation*}\label{syst-mut}
\begin{cases}
\partial _t u=\partial _{xx}u+u(1-(u+v))+\mu (v-u)\vspace{5pt}
\\
\partial _t v=\partial _{xx} v+rv\left(1-\displaystyle
\frac{u+v}K\right)+\mu (u-v).
\end{cases}
\end{equation*}
On the other hand, it is known that the spatial structure 
has a great influence on host-parasites systems, both at the epidemiological and evolutionary levels  \cite{Bes-Web-11}, \cite{Ash-Gup-14}, \cite{Lio-Gan-15}. In order to understand the  influence of heterogeneities, we aim at studying steady states and propagating solutions, or {\it fronts}, of system \eqref{syst}.

\medskip

\noindent{\it Traveling fronts in homogeneous environments.}  In a homogeneous environment, propagation in reaction diffusion equations is typically described by \textit{traveling waves}, namely solutions to the parabolic equation consisting of a constant profile shifting at a constant speed.  This goes back to the seminal works \cite{Fis-37}, \cite{Kol-Pet-Pis-37} on the  Fisher-KPP equation
\begin{equation*}
 \partial_t u=\Delta u+u(1-u),
\end{equation*}
a model for the spreading of advantageous
genetic features in a population. The literature on traveling fronts
for such homogeneous reaction diffusion equations is very large, see 
\cite{Fis-37}, \cite{Kol-Pet-Pis-37}, \cite{Aro-Wei-75, Aro-Wei-78}, \cite{Fif-Mac}, \cite{Gar82}, \cite{Ber-Nic-Sch85} among others. In such situations, many techniques based on the comparison
principle --- such as some monotone iterative schemes
 or the sliding method \cite{Ber-Nir-91}--- can be used to get {\it a priori} bounds, existence and monotonicity properties of  the solution.

Nevertheless, when considering nonlocal effects or systems, the comparison principle may no longer be available so that the above techniques do not apply and the situation is more involved. One usually uses topological degree arguments to construct traveling wave solutions: see
\cite{Ber-Nad-Per-Ryz}, \cite{Fan-Zha-11}, \cite{Alf-Cov-12}, \cite{Ham-Ryz-14} for
the nonlocal Fisher-KPP equation, \cite{Alf-Cov-Rao-14} for a
bistable nonlocal equation, \cite{Alf-Cov-Rao-13} for a nonlocal equation in an evolutionary context, \cite{Gri-Rao}
for a homogeneous system in an evolutionary context... Notice also that the boundary conditions are then typically understood in a weak sense, meaning that the wave connects 0 to \lq\lq something positive'' that cannot easily be identified: for example, in the nonlocal Fisher-KPP equation the positive steady state $u\equiv
1$ may present  a Turing instability.

\medskip

In a heterogeneous environment, however, it is unreasonable to expect the existence of such a solution.
The particular type of propagating solution we aim at constructing in our periodic case  is the so called  \textit{pulsating front}, first introduced by Xin 
\cite{Xin-00} in the framework of flame
propagation.

\medskip

\noindent{\it Pulsating fronts in heterogeneous environments.} The definition of a pulsating front is the natural extension, in the periodic
framework, of the aforementioned traveling waves. We introduce a speed $ c $ and shift the origin with this speed to catch the asymptotic dynamics.  Technically, a \textit{pulsating front} (with speed $c$) is then a profile $ (U(s, x), V(s,x)) $ that is periodic in the space 
variable $ x $, and that connects $ (0, 0) $ to a non-trivial state, such that $ (u(t,x), v(t,x)):=(U(x-ct, c), V(x-ct, x)) $ solves \eqref{syst}. Equivalently, a pulsating front is a 
solution of \eqref{syst} connecting $ (0, 0) $ to a non-trivial state, and that satisfies the constraint
$$
\left(u\left(t+\frac{L}{c}, x\right), v\left(t+\frac{L}{c}, x\right)\right)=(u(t, x-L), v(t, x-L)), \quad \forall (t,x)\in \R ^2.
$$
 As far as monostable pulsating fronts are
concerned, we refer among others to the seminal works of Weinberger \cite{Wei-02},
Berestycki and Hamel \cite{Ber-Ham-02}. Let us also mention \cite{HZ},
\cite{Ber-Ham-Roq-05-n2}, \cite{Ham-08}, \cite{Ham-Roq-11} for
related results.

\medskip

One of the main difficulties we encounter when studying system \eqref{syst} is that two main dynamics co-exist. 
On the one hand, when the population is small,  \eqref{syst} behaves like a cooperative system which enjoys a comparison principle.
On the other hand, when the population is near a non-trivial equilibrium, \eqref{syst}
is closer to a competitive system.
Since those dynamics cannot be separated, our system does not admit any comparison principle, and standard
techniques such as monotone iterations cannot be applied. As far as we know, the present work is the first construction  of pulsating fronts in a situation where comparison arguments are not available.

\section{Main results and comments}\label{s:results} 

\subsection{Assumptions, linear material and notations}\label{ss:assumptions}

\noindent {\bf Periodic coefficients.} Throughout this work, and even if not recalled, we always make the folllowing assumptions.
Functions $ r_u, r_v, \gamma_u, \gamma_v, \mu  :\mathbb R\rightarrow \mathbb R $ are smooth and periodic with period $ L>0$.  
We assume further that $ \gamma_u$, $ \gamma_v $ and $ \mu $ are positive. We denote their bounds
\begin{equation*}
    \begin{matrix}
       & 0<&\mingamma\leq    & \gamma_u(x), \gamma_v(x)      &\leq \maxgamma         \\
    & 0<&\minmu\leq       & \mu(x)                        &\leq \maxmu            \\
        &   &\minr\leq        & r_u(x),r_v(x)                 & \leq \maxr,
    \end{matrix}
\end{equation*}
for all $x\in \R$. Notice that $ r_u $ and $ r_v $ are allowed to take negative values, which is an additional difficulty, in particular in the proofs of Lemma 
\ref{lem:firsthomotopyaprioriestimates} and Lemma \ref{lem:cnotto0}. The fact that $ r_u, r_v $ do not have a positive lower bound is the main reason why  we need to introduce several types of eigenvalue problems, see \eqref{eq:eigendirper} and  \eqref{eq:stationaryDirichlet}, to construct subsolutions of related problems. 

\medskip 

\noindent {\bf On the linearized system around $(0,0)$.} 
We denote by $ A $  the symmetric matrix field arising after linearizing system \eqref{syst} near the trivial solution $ (0, 0) $, namely
\begin{equation}\label{eq:deflinearizedmatrix}
A(x):=\vectorize{r_u(x)-\mu(x) & \mu(x) \\ \mu(x) & r_v(x)-\mu(x)}.
\end{equation}
Since $ A(x) $ has positive off-diagonal coefficients, the elliptic system associated with the linear operator $-\Delta -A(x)$ is cooperative, {\it fully coupled} and therefore satisfies the strong maximum principle as well as other 
convenient properties \cite{Bus-Sir04}.

\begin{rem}[Cooperative elliptic systems and comparison principle] Cooperative systems enjoy similar comparison properties as scalar elliptic operators. In particular, \cite{Bus-Sir04} and \cite{Dam-Pac-12} show that 
the maximum principle holds for cooperative systems if the principal eigenvalue is positive. Moreover,  Section 13 (see also the beginning of Section
14) of \cite{Bus-Sir04} shows that, for so-called {\it fully coupled systems}
(which is the case of all the operators we will encounter since $\mu(x)\geq \mu ^0>0$),  the converse holds. These facts will be used for instance in the proof of Lemma \ref{lem:speed}.
\end{rem}

Let us now introduce a principal eigenvalue problem that is necessary to enunciate our main results.

\begin{defn}[Principal eigenvalue]\label{def:vp}
    We denote by $\l $ the principal eigenvalue of the stationary operator $ -\Delta-A(x) $ with periodic conditions, where $ A $ is defined in \eqref{eq:deflinearizedmatrix}.
\end{defn}
In particular, we are equipped through this work with a principal eigenfunction $\Phi:=\vectorize{\varphi \\ \psi} $ satisfying
\begin{equation}\label{vp-u}
\begin{cases}
    -\Phi_{xx}-A(x)\Phi=\l \Phi \\
    \Phi \text{ is } L\text{-periodic}, \quad \Phi  \text { is positive}, \quad \Vert \Phi \Vert_{\Lp\infty}=1.
\end{cases}
\end{equation} 

For more details on principal eigenvalue for systems, we refer the reader to \cite{Bus-Sir04}, in particular to Theorem 13.1 
(Dirichlet boundary condition)  which provides the principal eigenfunction.  
Furthermore, in the case of symmetric (self-adjoint) systems as the one we consider, the equivalent definition 
\cite[(2.14)]{Dam-Pac-12} provides some additional properties, in
particular that  the eigenfunction minimizes the Rayleigh quotient.

\medskip 
\noindent {\bf Function spaces.}
To avoid confusion with the usual function spaces, we denote the function spaces on a couple of functions with a bold font. Hence $ \Lp{p}(\Omega):=L^{p}(\Omega)\times L^p(\Omega) $ for $ p \in [1, \infty] $ and $ \Hq{q}(\Omega):=H^q(\Omega)\times H^{q}(\Omega) $ for $ q\in\mathbb N$ are equipped with the norms
$$
    \left\Vert \vectorize{u \\ v}\right\Vert_{\Lp{p}}:= \left\Vert\vectorize{\Vert u\Vert_{L^p} \\ \Vert v\Vert_{L^p}}\right\Vert_p,\quad \left\Vert\vectorize{u\\v}\right\Vert_{\Hq{q}}:=\left\Vert\vectorize{\Vert u\Vert_{H^q} \\ \Vert v\Vert_{H^q}}\right\Vert_{2} .
   $$
   Similarly,
$ \mathbf C^{\alpha, \beta}:=C^{\alpha, \beta}\times C^{\alpha, \beta}$ for 
$ \alpha\in\mathbb N $ and $ \beta\in[0, 1]$ is equipped with 
$ \left\Vert\vectorize{u \\ v}\right\Vert_{\mathbf C^{\alpha, \beta}}:=\max\left(\Vert u\Vert_{C^{\alpha, \beta}}, \Vert v\Vert_{C^{\alpha, \beta}}\right)$  and $ \mathbf C^\alpha:=\mathbf C^{\alpha, 0} $.
The subscript of those spaces denotes a restriction to a subspace : $ \Lpper{p} $, $\Hqper{q}$, $\Contper$, $ \Lipper $, $ \Cunper $ for $L$-periodic functions, 
$\Hq{1}_0$ for functions that vanish on the boundary, etc. Those function spaces are Banach spaces, and $ \Hq 1 $, $\Hqper 1 $, $ \Hq 1_0$, $ \Lp 2 $ and $ \Lpper 2 $ have a canonical Hilbert structure.

\subsection{Main results}\label{ss:results}

As well-known in KPP situations, the sign of the principal eigenvalue $\l$ is of crucial importance for the fate of the population: we expect extinction when $\l >0$ and propagation (hence survival) when $\l <0$. To confirm this scenario, we first study the existence of a nontrivial nonnegative steady state of problem \eqref{syst}, that is a  nontrivial nonnegative $L$-periodic solution to the system
\begin{equation}\label{eq:systemorig}
\left\{\begin{array}{l}
-p''=(r_u(x)-\gamma_u(x)(p+q))p+\mu (x)q-\mu (x)p \vspace{5pt} \\
-q''=(r_v(x)-\gamma_v(x)(p+q))q+\mu (x)p-\mu (x)q.
\end{array}\right.
\end{equation}

\begin{theo}[On nonnegative steady states]\label{th:steady}
If $\l > 0$ then $(0,0)$ is the only nonnegative steady
state of problem \eqref{syst}.

On the other hand, if $\l < 0$ then there exists a nontrivial positive steady
state $(p(x)>0,q(x)>0)$ of problem \eqref{syst}.
\end{theo}

Next we turn to the long time behavior of the Cauchy problem associated with \eqref{syst}. First, we prove extinction when the principal eigenvalue is positive.

\begin{prop}[Extinction]\label{prop:extinction} Assume $\l
>0$. Let a nonnegative and bounded initial condition $(u^0(x),v^0(x))$ be given. Then, any nonnegative solution $(u(t,x),v(t,x)))$ of \eqref{syst} starting from $(u^0(x),v^0(x))$ goes extinct exponentially fast as $t\to \infty$, namely
$$
\max\left(\Vert u(t,\cdot)\Vert _{L^\infty(\R)},\Vert v(t,\cdot)\Vert _{L^\infty(\R)}\right)=O(e ^{-\l t}).
$$
\end{prop}

The proof of Proposition \ref{prop:extinction} is rather simple so we now present it.
The cooperative parabolic system
\begin{equation}\label{syst-avec-compa}
\begin{cases}
\partial _t \bar u=\partial _{xx}\bar u+(r_u(x)-\mu (x))\bar u+\mu (x)\bar v\vspace{5pt}
\\
\partial _t \bar v=\partial _{xx} \bar v+(r_v(x)-\mu (x))\bar v+\mu (x)\bar u,
\end{cases}
\end{equation}
enjoys the comparison principle, see \cite[Theorem 3.2]{Fol-Pol-09}.  On the one hand, any nonnegative $(u(t,x),v(t,x))$ solution of \eqref{syst} is a subsolution of \eqref{syst-avec-compa}. On the other hand one can check that $(M\varphi(x)e^{-\l t},M\psi(x)e^{-\l t})$ --- with $(\varphi,\psi)$ the principal eigenfunction satisfying \eqref{vp-u}--- is a solution of  \eqref{syst-avec-compa} which is initially larger than $(u^{0},v^{0})$, if $M>0$ is sufficiently large. Conclusion then follows from the comparison principle.

\medskip

The reverse situation $\l <0$ is much more involved. Since in this case we aim at controlling the solution from below, the nonlinear term in \eqref{syst}
has to be carefully estimated. In order to show that the population does invade the whole line when $\l<0$, we are going to construct pulsating fronts for \eqref{syst}.

\begin{defn}[Pulsating front]\label{def:pul} A pulsating front for \eqref{syst} is a speed $c>0$ and a classical positive solution $(u(t,x),v(t,x))$ to \eqref{syst},
which satisfy the constraint
\begin{equation}\label{eq:proppuls}
 \vectorize{u(t+\frac Lc, x)\\v(t+\frac Lc,x)}=\vectorize{u(t, x-L)\\v(t, x-L)},\quad    \forall (t,x)\in\mathbb R^2,
\end{equation}
and supplemented with the boundary conditions
\begin{equation}\label{weak-boundary}
    \liminf_{t\to+\infty} \vectorize{u(t,x)\\v(t,x)}>\vectorize{0 \\ 0},\quad \lim_{t\to-\infty} \vectorize{u(t,x)\\v(t,x)}=\vectorize{0 \\ 0},
\end{equation}
locally uniformly w.r.t. $x$.
\end{defn}

Following \cite{Ber-Ham-Roq-05-n2}, we introduce a new set of variables that correspond to the frame of reference that follows the front propagation, that is
$ (s, x):=(x-ct, x)$. 
In these new variables, system \eqref{syst} transfers into
\begin{equation}\label{eq-puls}
\left\{\begin{array}{l}
-(u_{xx}+2u_{xs}+ u_{ss})-cu_s=(r_u(x)-\gamma_u(x)(u+v))u+\mu (x)v-\mu (x)u \vspace{5pt} \\
-(v_{xx}+2v_{xs}+ v_{ss})-cv_s=(r_v(x)-\gamma_v(x)(u+v))v+\mu (x)u-\mu (x)v,
\end{array}\right.
\end{equation}
and the constraint \eqref{eq:proppuls} is equivalent to the $ L $-periodicity in $ x $ of the solutions to \eqref{eq-puls}.
An inherent  difficulty to this approach is that the underlying elliptic operator, see the left-hand side member of system \eqref{eq-puls}, is degenerate. 
This requires to consider a regularization of the operator and to derive a series of {\it a priori} estimates that do not depend on the regularization, 
see \cite{Ber-Ham-02} or \cite{Ber-Ham-Roq-05-n2}. In addition to this inherent difficulty, the problem under consideration \eqref{syst} does not admit a 
comparison principle, in contrast with the previous results on pulsating fronts. Nevertheless, as in the traveling wave case, if we only require  boundary 
conditions in a weak sense --- see \eqref{weak-boundary} in Definition \ref{def:pul}--- then we can construct a pulsating front for 
\eqref{syst} when the underlying  principal eigenvalue is negative. This is the main result of the present paper since, as far as we know, this is the first construction 
of a pulsating front in a situation without comparison principle. 

\begin{theo}[Construction of a pulsating front]\label{th:pulsating} Assume $\l<0$.
    Then there exists a pulsating front solution to \eqref{syst}.
\end{theo}

As clear in our construction through the paper, the speed $c^*>0$ of the pulsating front of Theorem \ref{th:pulsating} satisfies the bound
$$
0<c^*\leq \bar c ^0:=\inf\{c\geq 0: \exists \lambda >0, \mu _{c,0}(\lambda)=0\},
$$
where $\mu_{c,0}(\lambda)$ is the first eigenvalue of the operator  
$$
    S_{c, \lambda,0} \Psi:=-\Psi_{xx}+2\lambda\Psi_x+\left[\lambda(c-\lambda)Id-A(x)\right]\Psi
$$
with $ L $-periodic boundary conditions. In previous works on pulsating fronts \cite{Wei-02}, \cite{Ber-Ham-02}, \cite{Ber-Ham-Roq-05-n2}, it is typically proved that $\bar c ^0$ is actually the minimal speed of pulsating fronts (and that faster pulsating fronts $c>\bar c ^0$ also exist).  Nevertheless, those proofs seem to rely deeply on the fact that pulsating fronts, as in Definition \ref{def:pul}, are increasing in time, which is far from obvious in our context without comparison. We conjecture that this remains true but, for the sake of conciseness, we  leave it as an open question.

\medskip

The paper is organized as follows. Section \ref{s:steady-state} is concerned with the proof of Theorem \ref{th:steady} on steady states. In particular the construction of nontrivial steady states requires an adaptation of some  bifurcations results \cite{Rab-71-1, Rab-71}, \cite{Cran-Rab-1971} that are recalled in Appendix, Section \ref{s:topo}. The rest of the paper is devoted to the proof of Theorem \ref{th:pulsating}, that is the construction of a pulsating front. We first consider in Section \ref{s:strip} an $\ep$-regularization of the degenerate problem \eqref{eq-puls} in a strip, where existence of a solution is proved by a Leray-Schauder topological degree argument.  Then, in Section \ref{s:fronts} we let the strip tend to $\R ^2$ and finally let the regularization $\ep$ tend to zero to complete the proof of Theorem \ref{th:pulsating}. This requires, among others,  a generalization to elliptic systems of a Bernstein-type gradient estimate performed in \cite{Ber-Ham-05}, which is proved in Appendix, Section \ref{s:bernstein}.

\section{Steady states}\label{s:steady-state}

This section is devoted to the proof of Theorem \ref{th:steady}. The main difficulty is to prove the existence of a positive steady state to \eqref{syst} when $\l <0$. To do so, we shall use the bifurcation theory introduced in the context of Sturm-Liouville problems by Crandall and Rabinowitz \cite{Cran-Rab-1971}, \cite{Rab-71-1, Rab-71}. Though an equivalent result may be obtained using a topological degree argument, this efficient theory  
shows clearly the relationship between the existence of solutions to the nonlinear problem and the sign of the principal eigenvalue 
of the linearized operator near zero. 

We shall first state and prove an independent theorem that takes advantage of the Krein-Rutman theorem in the context of a bifurcation originating from the 
principal eigenvalue of an operator. 
We will then use this theorem to show the link between the existence of a non-trivial positive steady state for \eqref{syst}, 
and the sign of the principal eigenvalue defined in \eqref{vp-u}.

\subsection{Bifurcation result, a topological preliminary}

We first prove a general  bifurcation theorem, interesting by itself,  which will be used as an end-point of the proof of Theorem \ref{th:steady}. It consists in a refinement of the results in \cite{Cran-Rab-1971}, \cite{Rab-71, Rab-71-1},
 under the additional assumption that the linearized operator satisfies the hypotheses of the Krein-Rutman
Theorem. Our contribution is to show that the set of nontrivial fixed points only \lq\lq meets''  $\mathbb R\times \{ 0 \} $ at point $ (\frac 1{\l(T)}, 0)$, with $ \l (T)$ the principal eigenvalue of the linearized operator $ T$. 

This theorem is independent from the rest of the paper and we will thus use a different set of notations.

\begin{theo}[Bifurcation under Krein-Rutman assumption]\label{theo:orientedbifurcation}
Let $E$ be a Banach space. Let $ C \subset E$ be a closed convex cone with nonempty interior $ Int\,C\neq\varnothing $ and of vertex 0, i.e.
such that $ C\cap -C=\{0\}$. Let 
\begin{equation*}
\begin{array}{rccl}
F:&\mathbb R\times E &\rightarrow& E \\
 & (\alpha, x)&\mapsto& F(\alpha, x)
 \end{array}
\end{equation*}
be a continuous and compact operator, i.e. $F$ maps bounded sets into relatively compact ones. Let us define
\begin{equation*}
\mathcal S:=\overline{\{(\alpha, x)\in \mathbb R\times E\bsl\{0\}:F(\alpha, x)=x\}}
\end{equation*}
the closure of the set of nontrivial fixed points of $ F$, and 
\begin{equation*}
\mathbb P_\mathbb R\mathcal S:=\{\alpha\in\mathbb R: \exists x\in C\bsl \{0\}, (\alpha, x)\in\mathcal S\}
\end{equation*}
the set of nontrivial solutions in $ C $.

 Let us assume the following.
\begin{enumerate}
\item \label{item:0isauniversalsolution} $ \forall \alpha\in\mathbb R$, $ F(\alpha, 0)=0 $. 
\item \label{item:frechetdifferentiability} $ F $ is Fr\'echet differentiable near $ \mathbb R\times \{0\} $ with derivative $ \alpha T $ locally uniformly w.r.t. $ \alpha $, i.e. for any $ \alpha_1< \alpha_2 $ and $ \epsilon>0 $ there exists $ \delta>0 $ such that  
\begin{equation*}
    \forall \alpha\in(\alpha_1, \alpha_2),\;  \Vert x\Vert\leq\delta\Rightarrow \Vert F(\alpha, x)-\alpha Tx\Vert\leq\epsilon \Vert x\Vert.
\end{equation*}
\item \label{item:kr} $ T $ satisfies the hypotheses of  Theorem \ref{theo:krein-rutman} (Krein-Rutman), i.e. $ T(C\bsl \{0\})\subset \mathrm{Int\,}{C}$. We denote by $\l (T)>0 $ its principal eigenvalue. 
\item  \label{item:theobiflocalbound} $ \mathcal S\cap (\{\alpha\}\times C) $ is bounded locally uniformly w.r.t. $ \alpha\in\R $.
\item \label{item:noescapethroughboundary} There  is no fixed point on the boundary of $ C $, i.e. $ \mathcal S\cap \mathbb (\mathbb R\times(\partial C\backslash\{0\}))=\varnothing$. 
\end{enumerate}

Then, either $ \left(-\infty, \frac{1}{\lambda_1(T)}\right)\subset\mathbb P_\mathbb R\mathcal S $ or $ \left(\frac{1}{\lambda_1(T)}, +\infty\right)\subset\mathbb P_\mathbb R\mathcal S$.
\end{theo}

\begin{proof} Let us first give a short overview of the proof. Since $ \l $ is a simple eigenvalue, we know from Theorem 
	\ref{theo:krasnoleski-rabinowitz} that there exists a branch of nontrivial solutions originating from 
	$ \left(\frac{1}{\l},0\right)$. We will show that this branch is actually contained in $ \mathbb R\times C $, thanks to 
	Theorem \ref{theo:rabinowitz}. Since it cannot meet $ \mathbb R\times \{0\} $ except at $ \left(\frac{1}{\l}, 0\right)$, it 
	has to be unbounded, which proves our result.

	Let us define 
$$
    \mathcal S_C:=\overline{\{(\alpha, x)\in\mathbb R\times(C\bsl\{0\}):F(\alpha, x)=x\}}
$$
which is a subset of $\mathcal S$, 
and $ \alpha_1:=\frac{1}{\lambda_1(T)} $. We may call $ (\alpha, x)\in\mathcal S_C $ a \textit{degenerate} solution if $ x\in\partial C $, and a \textit{proper} solution otherwise.

Our first task is to show that the only degenerate solution is $ \{(\alpha_1, 0)\} $. We first show  
$ \mathcal S_C\cap(\mathbb R\times \partial C)\subset \left\{(\alpha_1, 0)\right\} $.
Let $ (\alpha, x)\in\mathcal S_C\cap (\mathbb R\times \partial C) $ be given. By item \ref{item:noescapethroughboundary} we must have $ x=0 $.
Let $(\alpha_n, x_n)\to(\alpha, 0) $  such that $x_n\in C \setminus \{0\}$ and $ F(\alpha_n, x_n)=x_n $. 
Let us define $ y_n=\frac{x_n}{\Vert
x_n\Vert}\in C\setminus \{0\}$. On the one hand since $ y_n $ is a bounded sequence and $ T $ is a compact operator, up to an extraction the sequence $ (Ty_n) $ converges to some $ z$ which, by item \ref{item:kr}, must belong to $C$. On the other hand
\begin{equation*}
y_n=\frac{x_n}{\Vert x_n\Vert} = \alpha_nTy_n+\frac{F(\alpha_n,
x_n)-\alpha_nTx_n}{\Vert x_n\Vert}=\alpha z+o(1)
\end{equation*}
in virtue of items \ref{item:0isauniversalsolution} and \ref{item:frechetdifferentiability}, so that in particular $ z\neq 0$ and $\alpha\neq 0$.  Since $ y_n\rightarrow
\alpha z $ and $ Ty_n\rightarrow z $ we have $ z=\alpha
Tz $. Hence $z\in C\setminus \{0\}$ is an eigenvector for $T$ associated with the eigenvalue $\frac 1 \alpha$ so that Theorem \ref{theo:krein-rutman} (Krein-Rutman) enforces $\alpha =\frac 1 {\l (T)}=\alpha _1$.

Next we aim at showing the reverse inclusion, that is $ \left\{(\alpha_1, 0)\right\} \subset \mathcal S_C\cap(\mathbb R\times \partial C)$.
We shall use the topologic results of Appendix \ref{s:topo}, namely Theorem  \ref{theo:krasnoleski-rabinowitz} and Theorem 
\ref{theo:rabinowitz}. Let $ z\in C $ be the eigenvector of $ T $ associated with $ \lambda_1(T) $ such that $ \Vert z\Vert = 1 $, $ T^* $ the dual
of $ T $, and $ l\in E' $ the eigenvector\footnote{Let us recall that according to the Fredholm alternative, we have
$ \dim\ker(I-\lambda T)=\dim\ker(I-\lambda T^*)<\infty $ so that each eigenvalue of $ T $ is an eigenvalue of $ T^* $ with the 
same multiplicity.}  of $ T^* $ associated with $\lambda_1(T) $ such that $ \langle l, z\rangle = 1 $, where 
$ \langle\cdot, \cdot\rangle $ denotes the duality between $ E $ and its dual $ E'$.

Now, for $ \xi>0 $ and $ \eta\in(0, 1) $, let us define
\begin{equation*}
K_{\xi, \eta}^+:=\{(\alpha, x)\in\mathbb R\times
E:|\alpha-\alpha_1|<\xi, \langle l, x\rangle>\eta\Vert x\Vert\}.
\end{equation*}
The above sets are used to study the local properties of $ \mathcal S $ near the branching point $ (\alpha_1, 0)$. More precisely, 
it follows from Theorem \ref{theo:rabinowitz} that $ \mathcal S\bsl\{(\alpha_1, 0)\} $ contains a nontrivial connex
compound $ \mathcal C_{\alpha_1}^+ $ which is included in $ K_{\xi, \eta}^+ $ and 
near $ (\alpha_1, 0) $~:
\begin{equation*}
\forall \xi>0, \forall\eta\in(0, 1), \exists\zeta_0>0,\forall\zeta\in(0, \zeta_0), \, (\mathcal C_{\alpha_1}^+
\cap B_\zeta)\subset K_{\xi, \eta}^+,
\end{equation*}
where
\begin{equation*}
B_\zeta=\{(\alpha, x)\in\mathbb R\times E:|\alpha-\alpha_1|<\zeta,
\Vert x\Vert<\zeta\}.
\end{equation*}
Moreover, $ \mathcal C_{\alpha_1}^+ $ satisfies the alternative in Theorem
\ref{theo:krasnoleski-rabinowitz}. Let us show that $(\mathcal C_{\alpha_1}^+\cap B_\zeta)\subset \R \times C $ for $ \zeta>0 $ small enough, i.e.
\begin{equation}\label{letusprove}
    \exists \zeta>0, (\mathcal C_{\alpha_1}^+\cap B_\zeta)\subset \R\times C.
\end{equation}
To do so, assume by contradiction that  there exists a sequence $ (\alpha ^n, x_n)\to(\alpha_1, 0) $ such that
$$
    \forall n\in\mathbb N, (\alpha^{n}, x_n)\in \mathcal C_{\alpha_1}^+ \text{ and } x_n\notin C.
$$
Writing $ \frac{x_n}{\Vert x_n\Vert}=\alpha^{n}T\frac{x_n}{\Vert x_n\Vert}+\frac{F(\alpha^{n}, x_n)-\alpha^n Tx_n}{\Vert x_n\Vert}$ and reasoning as above, 
we see that (up to extraction) the sequence $\left(\frac{x_n}{\Vert x_n\Vert}\right)$ converges to some $w$ such that $Tw=\frac 1{\alpha _1}w=\l (T) w$. 
As a result $w=z$ or $w=-z$ (recall that $ z $ is the unique eigenvector of $ T $ such that $ z\in C $ and $ \Vert z\Vert = 1$). But the property $ \langle l, x_n\rangle \geq \eta\Vert x_n\Vert $ enforces $ \frac{x_n}{\Vert x_n\Vert}\to z $.
Since $ \frac{x_n}{\Vert x_n\Vert} \not\in C $ and $ z\in Int\,C $, this is a contradiction. Hence \eqref{letusprove} is proved. 

Since $ \mathcal C_{\alpha_1}^+ $ is connected and $ \mathcal C_{\alpha_1}^+\cap (\mathbb R\times\partial C) = \varnothing $ by item \ref{item:noescapethroughboundary}, we deduce from \eqref{letusprove} that $ \mathcal C_{\alpha_1}^+\subset \mathcal S_C $. Moreover, since by definition $ \{(\alpha_1, 0)\}\in \overline{\mathcal C_{\alpha_1}^+} $ and $ \mathcal S_C $ is closed, we have 
$$ 
    \left\{(\alpha_1, 0)\right\} \subset \mathcal S_C\cap(\mathbb R\times \partial C).
$$

We have then established that $ \{(\alpha_1, 0)\} $ is the only degenerate solution in $ C $ i.e. $ \mathcal S_C\cap(\mathbb R\times \partial C) = \{(\alpha_1, 0)\} $. 
Applying Theorem \ref{theo:rabinowitz} near $ \{(\alpha_1, 0)\} $, there exists a branch $ \mathcal C_{\alpha_1}^+ $ of solutions such that 
$ \{(\alpha_1, 0)\} \subset \overline{ \mathcal C_{\alpha_1}^+} $. By the above argument, $ \mathcal C_{\alpha_1}^+\subset \mathcal S_C $. Since
$ \mathcal C_{\alpha_1}^+ $ cannot meet $ \mathbb R\times\{0\} $ at 
$(\alpha, 0)\neq (\alpha_1, 0) $, it follows from Theorem \ref{theo:rabinowitz} that $ \mathcal C_{\alpha_1}^+ $  is unbounded.  It therefore follows from item 
\ref{item:theobiflocalbound} that there exists a sequence $ (\alpha^n, x^n)\in \mathcal C_{\alpha_1}^+ $ with 
$ |\alpha^n|\to\infty $. Since $ \mathcal C_{\alpha_1}^+ $ contains only proper solutions (i.e. $\mathcal C_{\alpha_1}^+\cap(\mathbb R\times \partial C)=\varnothing $), 
the projection $ P_\mathbb R(\mathcal C_{\alpha_1}^+) $ of 
$ \mathcal C_{\alpha_1}^+ $ on $ \mathbb R $ is included in $ \mathbb P_\mathbb R\mathcal S $. Finally, the continuity of the projection $ P_\mathbb R $ and the fact
that $ \mathcal C_{\alpha_1}^+ $ is connected show that either $ (\alpha_1, \alpha^n)\subset P_\mathbb R(\mathcal C_{\alpha_1}^+) $ or $ (\alpha^n, \alpha_1)\subset P_\mathbb R(\mathcal C_{\alpha_1}^+) $, depending on $ \alpha_1\leq \alpha^n $ or $ \alpha^n\leq \alpha_1 $. Letting $ n\to\infty $ proves Theorem \ref{theo:orientedbifurcation}. \end{proof}

\subsection{A priori estimates on steady states}\label{ss:aprioristeady}

In order to meet the hypotheses of Theorem \ref{theo:orientedbifurcation} in subsection \ref{ss:preli-proof}, we prove some \textit{a priori} estimates on stationary 
solutions. We have in mind to apply Theorem \ref{theo:orientedbifurcation} in the cone of nonnegativity of $ \Lp\infty(\mathbb R) $. 
Specifically, Lemma \ref{lem:uniformupperbound} will be used to meet item \ref{item:theobiflocalbound} (the solutions are locally bounded), and Lemma 
\ref{lem:nosolutionintheborder} will be used to meet item \ref{item:noescapethroughboundary} (there is no solution on the boundary of the cone).

\begin{lem}[Uniform upper bound]\label{lem:uniformupperbound}
    There exists a constant $C=C(\maxr, \maxmu, \mingamma) >0 $ such that any nonnegative periodic solution $(p,q)$  to \eqref{eq:systemorig} satisfies
$p(x)\leq C$ and $q(x) \leq C$, for all $x\in \R$.
\end{lem}

\begin{proof}
    Let $ \vectorize{p \\ q} $ be a solution to system \eqref{eq:systemorig}, so that 
\begin{equation} \label{eq:upperbound}
    \left\{\begin{array}{rcl}   -p'' & \leq & p(r_u-\gamma_u p) + q(\mu-\gamma_u p) \\
                                -q'' & \leq & q(r_v-\gamma_v q) + p(\mu-\gamma_v q).
    \end{array}\right.
\end{equation}
Let us define $ C:=\max\left(\frac{\maxr}{\mingamma}, \frac{\maxmu}{\mingamma}\right)>0$. Denote by $x_0$ a point where 
$p$ reaches its maximum, so that $-p''(x_0)\geq 0$. Assume by contradiction that $ p(x_0)> C $. Then, in virtue of \eqref{eq:upperbound}, one has 
$ -p''(x_0)\leq p(x_0)(r_u(x_0)-\gamma_u(x_0) C)< 0$, which is a contradiction. Thus $ p\leq C $. Inequality $q\leq C$ is proved the same way.
\end{proof}

\begin{lem}[Positivity of solutions]\label{lem:nosolutionintheborder}
Any nonnegative periodic solution  $ (p, q)$ to \eqref{eq:systemorig} such that $(p,q)\not\equiv (0,0)$ actually satisfies $p(x)>0$ and $q(x)>0$, for all $x\in \R$.
\end{lem}

\begin{proof} Write 
\begin{equation*} 
    \left\{\begin{array}{rcl}   -p'' & \geq & p(r_u-\mu-\gamma_u (p+q))  \\
                                -q'' & \geq & q(r_v-\mu-\gamma_v (p+q)),
    \end{array}\right.
\end{equation*}
   and the result is a direct application of the strong maximum principle.
\end{proof}

\subsection{Proof of the result on steady states}\label{ss:preli-proof}

We are now in the position to prove Theorem \ref{th:steady}.

\begin{proof}[The $\l >0$ case] Let $(p,q)$ be a nonnegative steady state solving \eqref{eq:systemorig}. We need to show that $(p,q)\equiv (0,0)$. Let us recall that $\Phi=\vectorize{\varphi \\ \psi} $ is the principal eigenfunction solving \eqref{vp-u}. From Lemma
\ref{lem:uniformupperbound}, we can define
\begin{equation}
    C_0:=\inf \left\{C\geq 0: \forall x\in \R, \vectorize{p(x) \\ q(x)}\leq C\vectorize{\varphi(x) \\ \psi(x)}\right\}.
\end{equation}
Let us assume by contradiction that $C_0>0$. Hence, without loss of generality, $p-C_0\varphi$ attains a zero maximum value at some point $x_0\in\R$, and $q-C_0 \psi \leq 0$ at this point. But, from \eqref{vp-u} and \eqref{eq:systemorig} we get
\begin{equation*}
\left\{\begin{array}{l}\label{eq:differenceupperestimatewitheigenvalues}
-(p-C_0\varphi)''-(r_u(x)-\mu (x))(p-C_0\varphi)=\mu (x)(q-C_0\psi)-\gamma_u(p+q)p-\l C_0\varphi< 0\vspace{5pt}\\
-(q-C_0\psi)''-(r_v(x)-\mu (x))(q-C_0\psi)=\mu (x)(p-C_0\varphi)-\gamma_v(p+q)q-\l
C_0\psi< 0.
\end{array}\right.
\end{equation*}
Evaluating the first inequality at point $x_0$ yields $(p-C_0\varphi)''(x_0)>0$, which is a contradiction since $ x_0 $ is a local maximum for $p-C_0\varphi$.  As a result $C_0=0$ and $(p,q)\equiv (0,0)$. 
\end{proof}

The reverse situation $\l<0$, where we need to prove the existence of a nontrivial steady state, is more involved. We shall combine our {\it a priori} estimates of 
subsection \ref{ss:aprioristeady} with our bifurcation result, namely Theorem \ref{theo:orientedbifurcation}. We will also use the $ \l >0 $ case. We want to stress 
eventually that we will use the notations introduced in subsection \ref{ss:assumptions}, in particular for functional spaces.

Before starting the proof itself, we would like to present briefly the core of the argument we use. We introduce a new parameter $ \beta\in\mathbb R $  and look at 
the modified system
\begin{equation}\label{eq:systembif}
    \left\{\begin{array}{rcl}   -p'' & = & p(r_u+\beta-\gamma_u(p+q))+\mu(q-p) \\
                                -q'' & = & q(r_v+\beta-\gamma_v(p+q))+\mu(p-q)
\end{array}\right.
\end{equation}
which is system \eqref{eq:systemorig} with $ r_u $ (resp. $ r_v $) replaced by $ r_u+\beta $ (resp. $ r_v+\beta $). We apply Theorem 
\ref{theo:orientedbifurcation} to system \eqref{eq:systembif} with the bifurcation parameter $ \beta $. There exists then a branch of
solutions originating from $ \beta=\l $, and which spans to $ \beta\to+\infty $ since the eigenvalue of the linearization of system \eqref{eq:systembif} is \textit{positive} for 
$\beta<\l $ (i.e. no solution exists for $\beta\in (-\infty, \l) $). In particular there exists a solution for $ \beta=0 $ since $ \l<0 $. Let us make this argument rigorous.

\begin{proof}[The $\l <0$ case] We start with the following lemma.

\begin{lem}[Fr\'echet differentiability]\label{lem:frechetdiff}
Let
\begin{equation*}
f\vectorize{p\\q}:=\left(\begin{matrix}-\gamma_u(p+q)p\\-\gamma_v(p+q)q\end{matrix}\right).
\end{equation*}
Then, the induced operator $ \Lpper\infty(\mathbb R)\longrightarrow\Lpper\infty(\mathbb R) $ is Fr\'echet
differentiable at $\vectorize{0 \\ 0}$ with derivative $0_{\Lp\infty}$.
\end{lem}

\begin{proof}
We need  to show that
$$
\left\Vert f\vectorize{p\\q}\right\Vert_{\Lpper\infty(\mathbb R)}=o\left(\left\Vert \vectorize{p\\q}\right\Vert_{\Lpper\infty(\mathbb R)}\right)
$$
as $ \left\Vert \vectorize{p\\q}\right\Vert_{\Lpper\infty(\mathbb R)} \to 0.
$
We have
\begin{equation*}
    \left\Vert f\vectorize{p\\q}\right\Vert_{\Lpper\infty(\mathbb R)} \leq  \maxgamma\left\Vert\vectorize{p\\q}\right\Vert_{\Lpper\infty(\mathbb R)}\Vert p+q\Vert_{L^\infty_{per}(\mathbb R)}\leq 2\maxgamma\left\Vert\vectorize{p\\q}\right\Vert_{\Lpper\infty(\mathbb R)}^2
\end{equation*}
which proves the lemma.
\end{proof}

We are now in the position to complete the proof of Theorem \ref{th:steady}. It follows from classical theory that, for $ M>0 $ large enough, the  problem
\begin{equation}\label{eq:defop}
	\left\{\begin{array}{l}
			-\vectorize{\tilde p \\ \tilde q}''-A(x) \vectorize{\tilde p \\ \tilde q} + M\vectorize{\tilde p \\ \tilde q} = \vectorize{p \\ q} \\
			\vectorize{\tilde p \\ \tilde q}\in \mathbf H^1_{per} 
		\end{array}\right.
\end{equation}
has a unique weak solution $\vectorize{\tilde p \\ \tilde q}$, for each $ \vectorize{p \\ q}\in\Lpper{2} $. Let us call $ L_M^{-1} $ the associated operator, namely
\begin{equation*}
	\begin{array}{rccl}
		L_M^{-1}:&\Lpper{2} & \rightarrow & \Hqper{1} \\
			 & \vectorize{p \\ q} & \mapsto & \vectorize{\tilde p \\ \tilde q}.
	\end{array}
\end{equation*}
Notice that, assuming $ M > -\lambda_1 $, the principal eigenvalue associated with problem \eqref{eq:defop} is $ \l':=\l+M>0 $, and recall that the actual algebraic eigenvalue  $\l(L_M^{-1}) $  of the operator $ L_M^{-1} $ is given by
$$	\lambda_1(L_M^{-1})=\frac{1}{\l '}>0. $$

From elliptic regularity, the restriction of $
L_M^{-1} $ to $ \Lpper\infty(\mathbb R) $ maps $ \Lpper\infty(\mathbb R) $ into $ \Lippertheta(\mathbb R) $, $0<\theta<1$,
and $ L_M^{-1} $ is therefore a compact
operator on $ \Lpper\infty(\mathbb R)$. Hence, 
\begin{equation*}
\begin{array}{rccl}
    F: & \R \times \Lpper\infty(\mathbb R) & \rightarrow & \Lpper\infty(\mathbb R) \\
       & \vectorize{\alpha, \vectorize{p\\q}} & \mapsto & L_M^{-1}\left(f\vectorize{p\\q}+\alpha\left(\begin{matrix}p\\q\end{matrix}\right)\right)
 \end{array}
 \end{equation*}
 is a continuous and compact map, to which we aim at applying Theorem \ref{theo:orientedbifurcation}. Let us recall that the cone of nonegativity
$$
C:=\left\{\vectorize{p\\q}\in \Lpper\infty(\mathbb R): \vectorize{p\\q}\geq\vectorize{0\\0}\right\}
$$
is, as required by Theorem \ref{theo:orientedbifurcation}, a closed convex cone of vertex $ 0 $ and nonempty interior in $ \Lpper\infty $.
Finally, we want to stress that solutions to $ F\left(\alpha, \vectorize{p \\ q}\right)=\vectorize{p \\ q} $ are classical solutions to the system
\begin{equation}\label{cestfini}
-\vectorize{p \\ q}''-A(x) \vectorize{p \\ q} = f\vectorize{p \\ q} + (\alpha - M) \vectorize{p\\q}
\end{equation}
which is equivalent to system \eqref{eq:systembif} with $\beta=\alpha -M$, where $\alpha$ is the bifurcation parameter. Let us check that all assumptions of Theorem \ref{theo:orientedbifurcation} are satisfied. 
\begin{enumerate}
    \item Clearly we have $ \forall \mathbb \alpha \in \R, F\vectorize{\alpha, \vectorize{0\\0}}=\vectorize{0\\0} $.

\item From Lemma \ref{lem:frechetdiff} and the composition rule for derivatives, $ F $ is Fr\'echet
    differentiable near $ \R \times \left\{\vectorize{0\\0}\right\} $ with derivative $\alpha L_M^{-1}$ locally uniformly w.r.t. $\alpha $. 

\item From the comparison principle (available for $L_M^{-1}$ since $ \l'>0 $, see \cite{Bus-Sir04}), $L_M^{-1} $  satisfies the hypotheses of the Krein-Rutman Theorem, namely $L_M^{-1}(C\setminus \{0\})\subset Int\, C$. 

\item Lemma \ref{lem:uniformupperbound} shows that, for any $ \alpha_*<\alpha^*$, $ \mathcal S\cap
    (\alpha_*, \alpha^*)\times C $ is bounded (in view of system \eqref{eq:systembif}, the constant $ C $ defined in the proof of Lemma \ref{lem:uniformupperbound} is locally bounded w.r.t. $ \alpha $).

\item From Lemma \ref{lem:nosolutionintheborder}, any nonnegative fixed point is positive, i.e. $
\mathcal S\cap\left(\mathbb R\times(\partial C\bsl\{0\})\right)=\varnothing$.
\end{enumerate}

We may now apply Theorem \ref{theo:orientedbifurcation} which
states that either $ \mathcal S\cap(\{\alpha\}\times (C\setminus\{0\}))\neq\varnothing $
for any $ \alpha\in(\l', +\infty)$ or $ \mathcal S\cap(\{\alpha\}\times (C\setminus\{0\}))\neq\varnothing $
for any $ \alpha\in(-\infty, \l')$. Invoking the case of positive principal eigenvalue 
(see the begininning of the present subsection), we see that there is no nonnegative nontrivial fixed points when $\alpha <\l' $. 
As a result we have
$$
    \forall\alpha\in(\l', +\infty),  \mathcal S\cap(\{\alpha\}\times (C\setminus\{0\}))\neq\varnothing .
$$  
In particular, since $ \l'=M+\l<M$,
there exists a positive fixed point for $ \alpha=M$, which is a classical solution of \eqref{cestfini}. This completes the proof of Theorem \ref{th:steady}.
\end{proof}

\section{Towards pulsating fronts: the problem in a strip}\label{s:strip}

We have established above the existence of a nontrivial periodic  steady state $(p(x)>0,q(x)>0)$ when the first eigenvalue of the linearized stationary problem $\l $ is
negative. The rest of the paper is devoted to the construction of a pulsating front, see Definition \ref{def:pul}, when $\l <0$.

In order to circumvent the degeneracy of the elliptic operator in \eqref{eq-puls} we need to introduce a regularization via a small positive parameter $\ep$. Also, in order to gain compactness,  the system \eqref{eq-puls} posed in $(s,x)\in \R^2$ (recall that $ s=x-ct $) is first reduced to a strip $(s,x)\in (-a,a)\times \R$ (recall the periodicity in the $x$ variable). 

More precisely, let us first define the constants $a_0^*>0$ (minimal size of the strip in the $s$ variable on which we impose a normalization), $\nu _0>0$ (maximal normalization), and  $K_0>0$ by
$$
a_0^* := 2\sqrt{\frac{5}{-\l}}, \qquad \nu_0 := \min\left(1, \frac{-\l}{4\maxgamma}, \underset{x\in\mathbb R}{\min}(p(x), q(x))\right),
$$
$$
K_0 := \max\left(\frac{8\maxgamma\max
_{x\in\R}(p(x)+q(x))}{-\l}, 1+\underset{x\in\mathbb R}{\max}\left(\frac{p(x)}{q(x)}, \frac{q(x)}{p(x)}\right)\right).
$$
Also we define the strip $\Omega_0:=(-a_0, a_0)\times\mathbb R$ for $ a_0\geq a_0^* $.

 \begin{thm}[A solution of the regularized problem in a strip]\label{thm:existencestrip}
Assume $ \l<0$. Let $ a_0> a_0^* $, $0<\nu<\nu_0$ and $ K> K_0$ be given.  Then there is $C>0$ such that, for any $ \varepsilon\in(0,1) $, there is $\bar a=\bar a^{\ep}>0$ (whose definition can be found in 
 Lemma \ref{lem:firsthomotopyaprioriestimates} item \ref{item:firsthomotopyupperboundc}) such that: for any  $ a\geq a_0+\bar a $,  there exist a $ L $-periodic in $x$ and positive $(u(s,x),v(s,x))$, bounded by $C$, and a speed $c\in (0, \bar c^{\ep}+\ep)$, solving the following mixed Dirichlet-periodic problem on the domain $ \Omega:=(-a, a)\times \mathbb R $
\begin{equation}\label{eq:pbexistencestrip}
\left\{\begin{array}{l}\begin{array}{rcl}
L_\ep u-cu_s&=&u(r_u-\gamma_u(u+v))+\mu v-\mu u \quad \text{ in } \Omega \vspace{3pt} \\
L_\ep v-cv_s&=&v(r_v-\gamma_v(u+v))+\mu u-\mu v \quad \text{ in } \Omega \vspace{3pt}
\end{array} \\
 (u, v)(-a,x )=(Kp(x), Kq(x)), \quad \forall x \in \R\vspace{3pt}\\
 (u, v)(a, x)=(0, 0), \quad \forall x\in \R\vspace{3pt}\\
\Normalizationleft{u+v}=\nu,
\end{array}\right.
\end{equation}
where  $L_\ep:=-\partial_{xx}-2\partial_{xs}-(1+\ep)\partial_{ss}$ and  the speed $\bar c ^{\ep}\geq 0 $ is defined in 
Lemma \ref{lem:speed}.
\end{thm}

This whole section is concerned with the proof of Theorem \ref{thm:existencestrip}. In order to use a topological degree argument, we transform continuously our 
problem until we get a simpler problem for which we know how to 
compute the degree explicitely.

Our first homotopy allows us to get rid of the competitive behaviour of the system. Technically we interpolate the nonlinear terms $ -\gamma_uuv $, $ -\gamma_v uv $ with the linear terms $-\gamma_u u \frac qK $, $ -\gamma_vv\frac pK$ respectively, to obtain system \eqref{eq:firsthomotopytauequalzero} which is truly 
cooperative. In particular, since the boundary condition at $ s=-a $ is a supersolution to \eqref{eq:firsthomotopytauequalzero}, we can prove the existence of a unique
solution to \eqref{eq:firsthomotopytauequalzero} for each $ c\in\mathbb R $ via a monotone iteration technique, the monotonicity of the constructed solutions and further properties. Nevertheless we still need to compute the degree explicitely, to which end we use a second homotopy that interpolates the right-hand side of  \eqref{eq:firsthomotopytauequalzero} with a linear term, and then a third homotopy to get rid of the coupling between the speed $ c $ and the profiles $ u $ and $ v $. At this point we are equipped 
to compute the degree. For related arguments in a traveling wave context, we refer teh reader to \cite{Ber-Nad-Per-Ryz}, \cite{Alf-Cov-Rao-13, Alf-Cov-Rao-14}, \cite{Gri-Rao}.

The role of the a priori estimates in  subsections \ref{ss:firsthomotop}, \ref{ss:endfirsthomotop} and \ref{ss:secondhomotop} is to ensure that there is no 
solution on the boundary of the open sets that we choose to contain our problem, and thus that the degree is a constant along our path. In subsection 
\ref{ss:proofexistencestrip}, we complete the proof of Theorem \ref{thm:existencestrip}.

Before that,  we need to establish some properties on the upper bound $\bar c^{\ep}$ for the speed in Theorem \ref{thm:existencestrip}.

\begin{lem}[On the upper bound for the speed]\label{lem:speed}
Let 
$$
    S_{c, \lambda, \varepsilon} \Psi:=-\Psi_{xx}+2\lambda\Psi_x+\left[\lambda(c-(1+\varepsilon)\lambda)Id-A(x)\right]\Psi,
$$
and define
\begin{equation}\label{eq:defminspeedeps}
    \bar c^\varepsilon = \inf\left\{c\geq 0, \exists \lambda > 0, \mu_{c, \varepsilon}(\lambda) = 0\right\},
\end{equation}
where $ \mu_{c, \varepsilon}(\lambda) $ is the first eigenvalue of the operator $S_{c, \lambda, \varepsilon}$ with $ L $-periodic boundary conditions. Then the following holds.
\begin{enumerate}
\item For any $ \varepsilon\in (0,1) $, we have $ \bar c^\varepsilon<+\infty $.
\item We have $ \bar c^\varepsilon=\min\left\{c\geq 0, \exists \lambda > 0, \mu_{c, \varepsilon}(\lambda) = 0\right\} $.
\item $ \varepsilon\mapsto \bar c^\varepsilon $ is nondecreasing.
\end{enumerate}
\end{lem}
\begin{proof}
\begin{enumerate}
\item We need to prove that the set in the right-hand side of \eqref{eq:defminspeedeps} is non-empty. We first notice that $ \mu _{c,\ep} (0)=\l<0 $ for any $c>0$. Next, for the eigenfunction $ \Phi := \left(\begin{matrix}\varphi \\ \psi\end{matrix}\right) $ solving \eqref{vp-u}, we have
$
    S_{c, \lambda, \varepsilon}\Phi=\l \Phi +2\lambda\Phi_x+\lambda(c-(1+\varepsilon)\lambda)\Phi$. In particular for $ \lambda=\frac{c}{2} $, we have
    $$
 S_{c,\frac c 2, \varepsilon}\Phi\geq (\l+\frac{c^{2}}{4}(1-\varepsilon))\Phi+c\Phi _x
\geq \left(\begin{matrix}0 \\ 0\end{matrix}\right) 
 $$
 as soon as $c\geq c_*$ where $c_*>0$ depends only on the quantities $\min(\varphi,\psi)$,  $\Vert\Phi_x\Vert_{\Lp\infty}$ and $-\l$. It therefore follows from  \cite[Theorem 13.1, item c]{Bus-Sir04} that $\mu _{c_*,\ep}\left(\frac{c_*}{2}\right)\geq 0$. Since the principal eigenvalue of $ S_{c, \lambda, \varepsilon} $ is 
continuous\footnote{This property is potentially false in general but has a simple proof in our setting. Take a sequence of operators $ T_n\to T $ that send a proper cone $ C $ into $ K\subset \mathrm{Int~}C $ with $ K $ compact, i.e. $ T_n(C)\subset K $ and $ T(C)\subset K $. Assume that the series of normalized eigenvectors $ x_n\in C $ s.t. $ T_nx_n=\lambda_nx_n $ diverges, then we can extract to sequences $ x_n^1\to y\in C $ and $ x_n^2\to z\in C $ with $ y\neq z $. Extracting further, there exists $\mu $ and $ \nu $ s.t. $ Ty=\mu y $ and $ T z=\nu z $ which is a contradiction since $ y\neq z $. Hence the continuity of the eigenvalue.} 
with respect to $ \lambda $ (and $c$), there exists  $ \lambda\in(0, \frac {c_*}2] $ such that $ \mu_{c_*, \varepsilon}(\lambda)=0 $, which proves that \eqref{eq:defminspeedeps} is 
well-posed.

\item  For the eigenfunction $ \Phi $ solving \eqref{vp-u}, we have
$$
S_{c, \lambda, \varepsilon}\Phi\leq 2 \lambda \Phi _x-\lambda^2\left(1+\varepsilon-\frac c \lambda \right)\Phi<\left(\begin{matrix}0 \\ 0\end{matrix}\right) 
$$
 as soon as $\lambda \geq \lambda _*$ where $\lambda _*>0$ depends only on $\min(\varphi,\psi)$,  $\Vert\Phi_x\Vert_{\Lp\infty}$, and an upper bound for $c$. Hence the maximum principle does not hold for $S_{c,\lambda,\ep}$, and it follows from  \cite[Theorem 14.1]{Bus-Sir04} that $\mu _{c,\ep}(\lambda) \leq 0$. \

Now, we consider  sequences $c_n\searrow \bar c ^\ep$, and $\lambda _n\geq 0$ such that $\mu _{c_n,\ep}(\lambda _n)=0$. From the above, we have $\lambda _n\leq \lambda _*$ so that, up to extraction, $\lambda _n\to \lambda _\infty$. From the continuity  of the principal eigenvalue, we deduce that $\mu _{\bar c ^{\ep},\ep}(\lambda _\infty)=0$, and the infimum in \eqref{eq:defminspeedeps} is attained.

\item Let $ \varepsilon'\leq\varepsilon $ and $ c>0 $ such that there is a positive solution $ \Theta  $ to $ S_{c, \lambda, \varepsilon}\Theta =\left(\begin{matrix}0 \\ 0\end{matrix}\right)  $. Then
$
    S_{c, \lambda, \varepsilon'}\Theta =(\varepsilon-\varepsilon')\lambda^2\Theta \geq \left(\begin{matrix}0 \\ 0\end{matrix}\right) 
$
so that, as in the proof of item 1, there exists $ 0<\lambda'\leq \lambda  $ such that $ \mu_{c, \varepsilon'}(\lambda') =0$. Thus
$$
\{c\geq 0, \exists \lambda > 0, \mu_{c, \ep}(\lambda)=0\}\subset\{c\geq 0, \exists\lambda > 0, \mu_{c, \ep'}(\lambda)=0\}.
$$
Taking the infimum on $ c $  yields $ \bar c^{\varepsilon'}\leq \bar c^{\varepsilon} $.
\end{enumerate}

Lemma \ref{lem:speed} is proved.
\end{proof}

\subsection{Estimates along the first homotopy} \label{ss:firsthomotop}

Let us recall that the role of the first homotopy is to get rid of the competition of our original problem ($\tau =1$), so that the classical comparison 
methods become available for $ \tau=0 $. Notice that it is crucial that the Dirichlet condition at $ s=-a $ is a supersolution for the $\tau =0$ problem, in order to apply a sliding 
method in the following subsection. Hence, for $0\leq \tau \leq 1$, we consider the problem
\begin{equation}\label{eq:firsthomotopy}
\left\{\begin{array}{l}
\begin{array}{rcl}
L_\varepsilon u-cu_s&=&u[r_u-\gamma_u(u+(\tau v+(1-\tau)\frac qK))]+\mu v-\mu u \vspace{3pt}\\
L_\varepsilon v-cv_s&=&v[r_v-\gamma_v((\tau u+(1-\tau)\frac pK)+v)]+\mu u-\mu v \vspace{3pt}
\end{array} \\
 (u, v)(-a,x )=(Kp(x), Kq(x)), \quad \forall x\in \R \vspace{3pt}\\
 (u, v)(a, x)=(0, 0), \quad \forall x\in \R,
\end{array}\right.
\end{equation}
along with the normalization condition $\Normalizationleft{u+v}=\nu$.

\begin{lem}[A priori estimates along the first homotopy]\label{lem:firsthomotopyaprioriestimates}
Let a nonnegative $ (u,v)\in\Homotopyspace $ (where $\Omega=(-a,a)\times \R$ and the periodicity is understood only w.r.t. the $x\in\R$ variable) and $c\in \R$ solve \eqref{eq:firsthomotopy}, with $0\leq  \tau\leq 1 $. Then 
\begin{enumerate}
\item $ (u, v ) $ is a classical solution to \eqref{eq:firsthomotopy}, i.e. $ (u,v)\in \mathbf C^{2}(\overline\Omega) $. 
\label{item:firsthomotopyregularity}
\item \label{item:firsthomotopyuniversalupperbound} The positive constant $ C:=\max(\frac{2\maxr}{\mingamma}, K\max(p+q)) $ is such that 
$$
 u(s,x)+v(s,x)\leq C, \quad \forall (s,x)\in \overline \Omega=[-a,a]\times \R.
$$
\item $ (u,v) $ is positive in $ \Omega $. 
\label{item:firsthomotopypositivity}
\item   Let $ \lambda _0>0$ and 
$ \Phi_0(x)=\left(\begin{matrix}\Phi_u(x)\\ \Phi_v(x)\end{matrix}\right)>\left(\begin{matrix}0\\0\end{matrix}\right) $ be such
that $ S_{\bar c^\varepsilon, \lambda _0, \varepsilon}\Phi_0 = 0 $ and $ \Vert \Phi _0\Vert_{\Lpper\infty(\mathbb R)} = 1$.
Define $ \bar a=\bar a ^\ep:=\max(-\frac {1}{\lambda_0}\ln\left(\frac{\nu\min(\Phi_u, \Phi_v)}{4K\max(p, q)}\right),1) $. Then if $ a\geq a_0+\bar a $ and $ c\geq \bar c^\varepsilon $, we have 
$ 
    \Normalizationleft{u+v}<\frac{\nu}{2}$.
\label{item:firsthomotopyupperboundc}
\item  If $ c=0 $ and $a\geq a_0+1$ then 
\begin{equation}\label{eq:conditionalnormalization}
    \Normalizationleft{u+v}\geq \frac{-\lep}{\maxgamma}-\frac{\max(p+q)}{K},
\end{equation}
where $ \lep $ is the principal eigenvalue of the operator $ L_\varepsilon - A(x) $ with Dirichlet condition in $s$ and $L$-periodic condition in $ x $, in the domain $ \Omega_0 $, as defined in \eqref{eq:eigendirper}.
\label{item:firsthomotopysubsolc=0}
\end{enumerate}
\end{lem}

\begin{proof}
\begin{enumerate}
\item This is true from classical elliptic regularity. We omit the details.

\item In view of \eqref{eq:firsthomotopy}, the sum $ S:=u+v$ satisfies
$$
\begin{array}{rcl}
L_\varepsilon S-cS_s&=&r_uu+r_vv-\gamma_u u(u+(1-\tau)\frac qK+\tau v)-\gamma_vv(v+(1-\tau)\frac pK+\tau u)\\
                    &\leq & \maxr S-\mingamma(u^2+v^2).
 \end{array}
$$
Since $ S^2=u^2+2uv+v^2\leq 2 (u^2+v^2) $, we have
$$
L_\varepsilon S-cS_s\leq  \frac{\mingamma}{2}S\left(\frac{2\maxr}{\mingamma}-S\right).
$$
Since the maximum principle holds for the operator $L_\varepsilon-c\partial_s $ independently of $ c $ and 
$\varepsilon>0 $, $ S $ cannot have an interior local maximum which is greater than $ \frac{2\maxr}{\mingamma} $. This along with the boundary conditions  $ S(-a,x)=K(p(x)+q(x)) $, $S(a,x)=0$ proves item \ref{item:firsthomotopyuniversalupperbound}.

\item Assume that there exists $ (s_0, x_0)\in (-a, a)\times \mathbb R $ such that $ u(s_0, x_0)=0 $. Since
$$
L_\varepsilon u-cu_s\geq u\left[r_u\left(x\right)-\gamma_u\left(x\right)\left(u+\left(\tau v+\left(1-\tau\right)\frac qK\right)\right)-\mu \left(x\right)\right],
$$
the strong maximum principle enforces $ u\equiv 0 $ which contradicts the boundary condition at $ s=-a $.
The same argument applies to  $ v $.

\item Let $ \zeta(s, x):=Be^{-\lambda _0 s}\Phi _0(x) $, $B>0$. Then we have
$$  
    L_{\varepsilon}\zeta - c \zeta_s = Be^{-\lambda _0 s}\left(S_{\bar c^\varepsilon, \lambda _0, \varepsilon} \Phi _0 + A(x)\Phi _0 + \lambda _0 (c-\bar c^\varepsilon)\Phi _0\right) = A(x) \zeta +\lambda _0(c-\bar c^\varepsilon)\zeta\geq A(x)\zeta
$$ 
so that $ \zeta $ is a strict supersolution to problem \eqref{eq:firsthomotopy}. By item \ref{item:firsthomotopyuniversalupperbound}, one can define
$$
    B_0:=\inf\left\{B>0, \forall (s, x)\in[-a,a]\times \mathbb R, \left(\begin{matrix} u(s,x)\\ v(s,x) \end{matrix}\right)\leq \zeta(s,x) \right\}>0
$$ 
and $ \zeta_0(s,x)=\left(\begin{matrix}\zeta_u(s,x) \\ \zeta_v(s,x)\end{matrix}\right):=B_0e^{-\lambda _0 s}\Phi _0(x) $. From the strong maximum principle in $(-a,a)\times \R$, and the $s=a$ boundary condition, the touching point has to lie on $s=-a$. 
 Thus there exists $ x_0 $ such that either $ \zeta_u(-a, x_0)=u(-a, x_0) $ or $\zeta_v(-a, x_0)=v(-a, x_0) $. In any case one has $B_0 \leq Ke^{-\lambda _0 a}\frac{\max(p, q)}{\min(\Phi _u, \Phi_v)} $, which in in turn implies
$$
    \Normalizationleft{u+v} \leq 2B_0e^{\lambda _0 a_0} \leq 2K\frac{\max(p, q)}{\min(\Phi_u, \Phi_v)}e^{-\lambda _0 (a-a_0)} \leq 2K\frac{\max(p, q)}{\min(\Phi_u, \Phi_v)}e^{-\lambda _0 \bar a}\leq \frac{\nu}{2},
$$
in view of the definition of $\bar a$. This  proves item \ref{item:firsthomotopyupperboundc}.

\item 
Assume by contradiction that
$
    \Normalizationleft{u+v}<\frac{-\lep}{\maxgamma}-\frac{\max(p+q)}{K}
$ (which in particular enforces $\lep <0$).
Then, in $ (-a_0, a_0)\times \R$, we have 
$$
    \left\{\begin{array}{rcl}
    L_\varepsilon u & = & (r_u-\mu -\gamma_u(u+\tau v+(1-\tau)\frac qK))u+\mu  v \geq (r_u-\mu +\lep)u+\mu  v \\
    L_\varepsilon v & = & (r_v-\mu -\gamma_v(v+\tau u+(1-\tau)\frac pK))v+\mu  u \geq (r_v-\mu +\lep)v+\mu  u .
    \end{array}\right.
$$
Denote by $\Phi ^\ep(s,x):=\vectorize{\bar \varphi(s,x)\\ \bar \psi(s,x)}$ the principal eigenvector associated with $\lep$ (vanishing at $s=\pm a_0$, $L$ periodic in $x$) normalized by $ \Vert\bar\Phi ^\ep\Vert_{\Lpper\infty(\mathbb R)}  = 1 $, see problem \eqref{eq:eigendirper}.  Define
$$
A_0:=\max\{A>0:  A\bar \varphi(s,x)\leq u(s,x) \text{ and }  A\bar \psi(s,x)\leq v(s,x), \forall (s,x)\in [-a_0,a_0]\times \R \}.
$$
 Then we have $A_0\bar \varphi\leq u$, $A_0\bar \psi \leq v$, with equality at at least one point for at least one equation, say $A_0\bar \varphi (s_0,x_0)=u(s_0,x_0)$ for some $
-a_0<s_0<a_0$ and $x_0\in \R$. But 
$$
L_\varepsilon (u-A_0\bar \varphi) - (r_u-\mu +\lep)(u-A_0\bar \varphi) \geq  \mu  (v-A_0\bar \psi)\geq 0,
$$
so that the  strong maximum principle enforces $u\equiv A_0\bar \varphi$, which is a contradiction since $ u $ is positive on $ (-a, a)\times \R $ and $ \bar\varphi $ vanishes on $ \{\pm a_0\}\times \R $. A similar argument leads to a contradiction in the case $ v(s_0, x_0)=A_0\bar\psi(s_0, x_0) $. This proves item \ref{item:firsthomotopysubsolc=0}. 
\end{enumerate}
Lemma \ref{lem:firsthomotopyaprioriestimates} is proved.
\end{proof}

Item \ref{item:firsthomotopysubsolc=0} of the above lemma is relevant only  when $ \lep<0 $, which is actually true if $a_0>0$ is large enough, as proved below. Let us denote by $\lep$, $\Phi ^\ep(s,x)$ the principal eigenvalue, eigenfunction solving the  mixed Dirichlet-periodic eigenproblem
\begin{equation}\label{eq:eigendirper}
    \left\{\begin{array}{l}
            L_\varepsilon \Phi^\ep=A(x)\Phi^\ep + \lep \Phi^\ep  \quad \text{ in } \Omega _0=(-a_0,a_0)\times \R\\
            \Phi^\ep(-a_0, x)=\Phi^\ep(a_0, x)=0\quad \forall x \in \R \\
            \Phi^\ep(s,x) \quad \text{ is periodic w.r.t. $x$}\\
            \Phi^\ep>\vectorize{0\\0} \text{ in } \Omega _0=(-a_0,a_0)\times \R.
        \end{array}
    \right.
\end{equation}

\begin{lem}[An estimate for $ \lep $]\label{lem:lepestimate}
 We have 
$\lep\leq\l+\frac{5}{2a_0^2}(1+\varepsilon)$.
\end{lem}
\begin{proof}
Since the matrix $A(x)$  is symmetric, we are equipped with the Rayleigh
quotient 
$$
\lep=\inf_{w\in H^1_{0,per}\times H^1_{0,per}}\frac{\int_{(-a_0, a_0)\times (0,L)}\left(\,^tw_xw_x+2\,^tw_xw_{s}+(1+\varepsilon)\,^tw_sw_{s}-\,^twA(x)w\right)\,\mathrm ds\mathrm dx}{\int_{(-a_0, a_0)\times (0,L)}
\,^tww\,\mathrm ds\mathrm dx}.
$$
Let us denote $ \Phi(x)=\left(\begin{matrix}\,\varphi(x) \\
\psi(x)\end{matrix}\right) $ the  principal eigenvector solving \eqref{vp-u}, and 
define $$ 
\bar\Phi:=\Vert\Phi\Vert_{\Lpper 2}^{-1}\Phi.
$$
We define the test function $ w(s, x):=\eta(s)\bar\Phi(x)$, with $\eta(s):=\sqrt{\frac{15}{16a_0^5}}(a_0-s)(a_0+s)$, so that $\int _{(-a_0,a_0)} \eta  ^{2}(s)ds=1$.
Noticing that $\int \,^tw_xw_{s}\mathrm dx\mathrm ds=0$,
we  get
$$
\lep\leq
\int_{(0,L)} (\,^t\bar\Phi_x\bar\Phi_x-\,^t\bar\Phi A(x)\bar\Phi)(x)\mathrm dx + \int_{(-a_0, a_0)}(1+\varepsilon)\eta_s^2(s)\mathrm ds=\l+\frac{5}{2a_0^2}(1+\varepsilon),
$$
which  shows the result.
\end{proof}

\begin{rem}[Consistency of the choice of parameters in Theorem \ref{thm:existencestrip}]\label{rem:consis} Let us say a word on the choice of the positive parameters  ($a_0^*$, $\nu_0$, $K_0$) in Theorem \ref{thm:existencestrip}. First, the choice of $a_0^*$ and Lemma \ref{lem:lepestimate} imply that $\lep \leq \frac{3\l}{4}$ for any $\ep \in(0,1)$ and $ a_0\geq a_0^* $. Then, \eqref{eq:conditionalnormalization} and the choices of $K_0$, $\nu _0$ 
imply that, for $c=0$,  
\[ 
\Normalizationleft{u+v}\geq
\frac{-\l}{2\maxgamma} \geq 2\nu_0.
\]
 In particular, item \ref{item:firsthomotopysubsolc=0} in Lemma \ref{lem:firsthomotopyaprioriestimates} gives a true lower bound for $ \Normalizationleft{u+v} $ in the case $ c=0 $.
\end{rem}

\subsection{Estimates for the end-point $\tau =0$ of the first homotopy}\label{ss:endfirsthomotop}

We introduce the problem
\begin{equation}\label{eq:firsthomotopytauequalzero}
\left\{\begin{array}{l}
\begin{array}{rcl}
L_\varepsilon u-cu_s&=&u(r_u-\gamma_u(u+\frac qK))+\mu v-\mu u \vspace{3pt}\\
L_\varepsilon v-cv_s&=&v(r_v-\gamma_v(\frac pK+v))+\mu u-\mu v\vspace{3pt} 
\end{array} \\
(u, v)(-a,x )=(Kp(x), Kq(x)), \quad \forall x\in \R \vspace{3pt}\\
 (u, v)(a, x)=(0, 0),\quad \forall x\in \R, 
\end{array}\right.
\end{equation}
which corresponds to \eqref{eq:firsthomotopy} with $ \tau=0 $ and for which comparison methods are available. In this subsection we derive refined estimates for \eqref{eq:firsthomotopytauequalzero} that will allow us 
to enlarge  the domain on which the degree is computed, which is necessary for the second homotopy that we will perform.

\begin{lem}[On problem \eqref{eq:firsthomotopytauequalzero}]\label{lem:firsthomotopytauequalzero}
\begin{enumerate}
    \item For each $ c\in\mathbb R $, there exists a unique nonnegative solution $(u,v)$ to \eqref{eq:firsthomotopytauequalzero}, which satisfies
\begin{equation}\label{bornes}
    \forall (s, x)\in\Omega, \quad 0< u(s, x)<Kp(x)~~\mathrm{and}~~0<v(s, x)<Kq(x).
\end{equation}
\item Let $ c\in\mathbb R $ and $(u,v)$ the nonnegative solution to \eqref{eq:firsthomotopytauequalzero}. Then $ u $ and $ v $ are nonincreasing in $ s $.
\item The mapping $ c\mapsto \vectorize{u\\v} $ is decreasing, where $ (u, v) $ is the unique nonnegative solution to  \eqref{eq:firsthomotopytauequalzero}.
\end{enumerate}
\end{lem}

\begin{proof}
In this proof we denote 
\begin{equation}\label{eq:prooftauequalzerononlinearfunction}
    f:\left(x,\left(\begin{matrix}u \\ v\end{matrix}\right)\right)\mapsto\left(\begin{matrix}u(r_u(x)-\gamma_u(x)(u+\frac qK))+\mu (x)v-\mu (x)u \\
v(r_v(x)-\gamma_v(x)(\frac pK+v))+\mu (x)u-\mu (x)v \end{matrix}\right)
\end{equation}
so that \eqref{eq:firsthomotopytauequalzero} is recast $L_\ep \left(\begin{matrix}u \\ v\end{matrix}\right)-c\left(\begin{matrix}u \\ v\end{matrix}\right)_s=f\left(x,\left(\begin{matrix}u \\ v\end{matrix}\right)\right)$. We select $M>0$ large enough so that $f(x,\cdot)+M Id$ is uniformly nondecreasing on $[0,C]^{2}$, with $C$ the constant from Lemma \ref{lem:firsthomotopyaprioriestimates}, that is
$$
\left(\begin{matrix}0 \\ 0\end{matrix}\right) \leq \left(\begin{matrix}u_1 \\ v_1\end{matrix}\right) \leq \left(\begin{matrix}u_2 \\ v_2\end{matrix}\right) \leq \left(\begin{matrix}C \\C\end{matrix}\right)\Rightarrow f\left(x,\left(\begin{matrix}u_2 \\ v_2\end{matrix}\right)\right)-f\left(x,\left(\begin{matrix}u_1 \\ v_1\end{matrix}\right)\right)\geq -M \left(\begin{matrix}u_2-u_1 \\ v_2-v_1\end{matrix}\right),
$$
for all $x\in \R$.

\begin{enumerate}
\item We first claim that $(s,x)\mapsto (Kp(x), Kq(x)) $ is a strict supersolution to problem \eqref{eq:firsthomotopytauequalzero}. Since $K\geq K_0$, we have $p+q<Kp\leq Kp+\frac q K$ so that
$$
\begin{array}{l}
\begin{array}{rcl}
L_\varepsilon (Kp)-c(Kp)_s&=&-(Kp)''\\
&=&(Kp)(r_u(x)-\gamma_u(x)(p+ q))+\mu (x)Kq-\mu (x)Kp  \\
    & >& (Kp)(r_u(x)-\gamma_u(x)(Kp+ \frac qK))+\mu (x)(Kq)-\mu (x)(Kp),
    \end{array}
    \end{array}
    $$
   and similarly 
$$
\begin{array}{l}
\begin{array}{rcl}
L_\varepsilon (Kq)-c(Kq)_s
    & > & (Kq)(r_v(x)-\gamma_v(x)(\frac pK+Kq))+\mu (x)(Kp)-\mu (x)(Kq),
\end{array} 
\end{array}
$$
which proves the claim. Obviously, $ (s,x)\mapsto \vectorize{0\\0} $ is a strict subsolution to problem \eqref{eq:firsthomotopytauequalzero} because of the boundary condition at $ s=-a $. Since system \eqref{eq:firsthomotopytauequalzero} is cooperative,
the classical monotone iteration method shows that, for any $c\in \R$,  there exists at least a solution $(u,v)$ 
to problem \eqref{eq:firsthomotopytauequalzero} which  satisfies \eqref{bornes}.

Next, in order to prove uniqueness, let $(u,v)$ and $ (\tilde u, \tilde v) $ be two nonnegative solutions to \eqref{eq:firsthomotopytauequalzero},
such that $ (u, v)\neq(\tilde u, \tilde v) $. Then, for any $0< \zeta <1$, $(U^\zeta, V^\zeta):=(\zeta u, \zeta v) $ satisfies
$$
    \left\{\begin{array}{l}\begin{array}{rcl}
    L_\varepsilon U^\zeta - cU^\zeta_s&=&U^\zeta(r_u-\gamma_u\frac qK-\mu -\frac {\gamma_u(x)}{\zeta}U^\zeta)+\mu (x)V^\zeta  \\
                                      &<&   U^\zeta(r_u-\gamma_u\frac qK-\mu -\gamma_u(x)U^\zeta)+\mu (x)V^\zeta           \\
    L_\varepsilon V^\zeta - cV^\zeta_s&=&V^\zeta(r_v-\gamma_v\frac pK-\mu -\frac {\gamma_v(x)}{\zeta}V^\zeta)+\mu (x)U^\zeta  \\
                                      &<&   V^\zeta(r_v-\gamma_v\frac pK-\mu -\gamma_v(x)V^\zeta)+\mu (x)U^\zeta
    \end{array}\\
    (U^\zeta, V^\zeta)(-a,x)=(\zeta Kp(x), \zeta Kq(x))\leq (Kp(x), Kq(x)) \\
    (U^\zeta, V^\zeta)(a,x)=(0, 0),
    \end{array}
    \right.
$$
and is therefore a strict subsolution to problem \eqref{eq:firsthomotopytauequalzero}. From Hopf lemma
we know that $ (\tilde u_s, \tilde v_s)(a, x) < (0, 0) $ so that we can define 
$$
    \zeta_0:=\sup \{\zeta>0: (U^\zeta, V^\zeta)(s,x)< (\tilde u, \tilde v)(s,x), \forall (s,x) \in \Omega\}>0.
$$ 
Then we have $ (0,0)\leq (U^{\zeta_0}, V^{\zeta_0})\leq (\tilde u, \tilde v)\leq (C, C) $. Assume by contradiction that $ \zeta_0<1 $. Then we have 
$$
    \left\{\begin{array}{l}\begin{array}{rcl}
    L_\varepsilon (\tilde u-U^{\zeta_0})-c(\tilde u-U^{\zeta_0})_s+M(\tilde u-U^{\zeta_0})&\geq& 0 \\
    L_\varepsilon (\tilde v-V^{\zeta_0})-c(\tilde v-V^{\zeta_0})_s+M(\tilde v-V^{\zeta_0})&\geq& 0 \\
    \end{array} \\
    (\tilde u-U^{\zeta_0},\tilde v-V^{\zeta_0})(-a,x)\geq (0, 0) \\
    (\tilde u-U^{\zeta_0},\tilde v-V^{\zeta_0})(a,x)= (0, 0).
    \end{array}
    \right.
$$
From Hopf lemma we deduce
$$
    ((\tilde u-U^{\zeta_0})_s,(\tilde v-V^{\zeta_0})_s)(a,x) < (0, 0)
$$
so that there exists $ (s_0, x_0)\in(-a, a)\times\mathbb R $ such that, say, $ \tilde u(s_0,x_0)=U^{\zeta_0}(s_0, x_0) $. From the strong maximum principle we deduce  
$  \tilde u\equiv U^{\zeta_0} $, which is a contradiction in view of the boundary condition at $s=-a$.
We conclude that $ \zeta_0\geq 1 $ and thus 
$ (u, v)\leq (\tilde u, \tilde v) $. Then exchanging the roles of $ (u, v) $ and $ (\tilde u, \tilde v) $ in the above argument, we get 
that $ (\tilde u, \tilde v)\leq (u, v) $ so that finally $ (\tilde u, \tilde v) = (u, v) $. This is in contradiction with our initial hypothesis.
We conclude that the nonnegative solution to equation \eqref{eq:firsthomotopytauequalzero} is unique.

\item For given $c\in \R$,  let $ (u, v) $ be the solution to \eqref{eq:firsthomotopytauequalzero}. In order to use a sliding technique, we define 
$$ 
(u^t(s,x), v^t(s, x)) := (u(s+t, x), v(s+t, x)) 
$$
for $ t>0 $ and $ (s, x)\in [-a, a-t]\times \mathbb R $. From the boundary conditions, there is $ \delta>0 $ such that 
$$
    \forall t\in (2a-\delta, 2a), \forall (s, x)\in(-a, a-t)\times\mathbb R, \quad u^t(s, x)< u(s, x)~\textrm{and}~ v^t(s, x) < v(s, x).
$$ 
In particular, one can define
$$ 
    t_0:=\inf\{t>0, \forall (s, x)\in [-a, a-t], \; u^t(s, x)\leq u(s, x)~\textrm{and}~ v^t(s, x) \leq v(s, x)\}.
$$
Assume by contradiction that $ t_0 > 0 $. Then there exists $(s_0, x_0)\in(-a, a-t_0)\times \mathbb R $ such that, say, $ u^{t_0}(s_0, x_0)=u(s_0, x_0) $  (notice that $s_0=-a$ and $s_0=a-t_0$ are prevented by \eqref{bornes}). Since we have
$$ 
        L_\varepsilon\vectorize{ u^{t_0}-u \\  v^{t_0}-v}-c\vectorize{ u^{t_0}-u \\  v^{t_0}-v}_s+M\vectorize{ u^{t_0}-u \\  v^{t_0}-v}=(f+M)\vectorize{ u^{t_0} \\  v^{t_0}} - (f+M)\vectorize{ u \\ v } \leq 0
        $$
        and $
        \vectorize{ u^{t_0}-u \\  v^{t_0}-v}\leq 0$, 
the strong maximum principle implies $u^{t_0}\equiv u$, which contradicts $0<u<Kp$. We conclude that $ t_0=0 $, which  means that $ u $ and $ v $ are nonincreasing in $ s $.

\item Let $(c, u, v) $ and $ (\tilde c, \tilde u, \tilde v) $ two solutions of equation \eqref{eq:firsthomotopytauequalzero} with $ c < \tilde c $. As above, we define 
$$ 
(\tilde u^t(s,x), \tilde v^t(s, x)) := (\tilde u(s+t, x), \tilde v(s+t, x)), 
$$
and
$$ 
    t_0:=\inf\{t>0, \forall (s, x)\in [-a, a-t],\; \tilde u^t(s, x)\leq u(s, x)~\textrm{and}~\tilde v^t(s, x) \leq v(s, x)\}.
$$
Assume by contradiction that $ t_0 > 0 $. Then there again exists $ (s_0, x_0)\in(-a, a-t_0)\times \mathbb R $ such that, say,   $ \tilde u^{t_0}(s_0, x_0)=u(s_0, x_0) $. Moreover we have
$$ 
   \begin{array}{lcl}
        && L_\varepsilon\vectorize{\tilde u^{t_0}-u \\ \tilde v^{t_0}-v}-c\vectorize{\tilde u^{t_0}-u \\ \tilde v^{t_0}-v}_s+M\vectorize{\tilde u^{t_0}-u \\ \tilde v^{t_0}-v} \\
         &&= (f+M)\vectorize{\tilde u^{t_0} \\ \tilde v^{t_0}} - (f+M)\vectorize{ u \\ v } 
                                                                                                                                                    +(\tilde c-c) \vectorize{ \tilde u ^{t_0} \\ \tilde v^{t_0}}_s\\
        & &\leq \vectorize{ 0 \\ 0 },
    \end{array}
$$
since $ \tilde u_s\leq 0 $ and $ \tilde v_s\leq 0$ (recall that $ \tilde u $ and $ \tilde v $ are decreasing), so that we again derive a contradiction. As a result 
 $ t_0=0 $ , that is  $ \left(\begin{matrix}\tilde u \\ \tilde v\end{matrix}\right)\leq\left(\begin{matrix}u\\v\end{matrix}\right)$ and then  $ \left(\begin{matrix}\tilde u \\ \tilde v\end{matrix}\right)< \left(\begin{matrix}u\\v\end{matrix}\right)$ from the strong maximum principle.
\end{enumerate}
The lemma is proved.
\end{proof}

\subsection{Estimates along the second homotopy}\label{ss:secondhomotop}

The second homotopy allows us to get rid of the nonlinearity and the coupling in $ u $ and $ v $ at the expense of an increased linear part. For $0\leq \tau\leq 1$, we consider
\begin{equation}\label{eq:secondhomotopy}
\left\{\begin{array}{l}\begin{array}{rcl}
L_\varepsilon u-c u_s&=&\tau\left(u\left(r_u-\gamma_u\frac qK-\mu -\gamma_u u \right)+\mu  v\right)-(1-\tau)\mathcal Cu \\
L_\varepsilon v-c v_s&=&\tau\left(v\left(r_v-\gamma_v\frac pK-\mu -\gamma_v v \right)+\mu  u\right)-(1-\tau)\mathcal Cv
\end{array} \\
(u, v)(-a,x )=(Kp(x), Kq(x)), \quad \forall x\in \R \vspace{3pt}\\
 (u, v)(a, x)=(0, 0),\quad \forall x\in \R, 
\end{array}\right.
\end{equation}
with
\begin{equation}\label{eq:secondhomotopylinearpenality}
    \mathcal C := -\underset{x\in\mathbb R}{\min}\left(r_u(x)-\gamma_u(x)\left(\frac{q(x)}{K}+ C\right) -\mu (x),r_v(x)-\gamma_v(x)\left(\frac{p(x)}{K}+C\right) -\mu (x),
    0\right)
\end{equation}
where $ C$ is as in Lemma \ref{lem:firsthomotopyaprioriestimates} item \ref{item:firsthomotopyuniversalupperbound}.

\begin{lem}[A priori estimates along the second homotopy]\label{lem:secondhomotopyaprioriestimates}
Let a nonnegative $ (u,v)\in\Homotopyspace $ (where $\Omega=(-a,a)\times \R$ and the periodicity is understood only w.r.t. the $x\in\R$ variable) and $c\in \R$ solve \eqref{eq:secondhomotopy}, with $0\leq  \tau\leq 1 $. Then 
\begin{enumerate}
\item $ (u, v ) $ is a classical solution to \eqref{eq:secondhomotopy}, i.e. $ (u,v)\in \mathbf C^{2}(\overline\Omega) $. 
\label{item:secondhomotopyregularity}
\item \label{item:secondhomotopyuniversalupperbound} We have
$$
 u(s,x)+v(s,x)\leq C, \quad \forall (s,x)\in \bar \Omega=[-a,a]\times \R.
$$
\item $ (u,v) $ is positive in $ \Omega $. 
\label{item:secondhomotopypositivity}
\item  If $ a\geq a_0+\bar a $ and $ c\geq \bar c^\varepsilon $, we have 
$ 
    \Normalizationleft{u+v}<\frac{\nu}{2}$, where $\bar a$ is as in Lemma \ref{lem:firsthomotopyaprioriestimates} item \ref{item:firsthomotopyupperboundc}.
\label{item:secondhomotopyupperboundc}
\item  There exists $ \underline c=\underline c(a)\geq0 $ such that if 
$ c\leq -\underline c(a) $ then
$
    \Normalizationleft{u+v} > \nu$. 
\label{item:secondhomotopysubsolc=0}
\end{enumerate}
\end{lem}

\begin{proof}
Items 1, 2, 3 and 4 can be proved as in Lemma \ref{lem:firsthomotopyaprioriestimates}. We therefore omit the details, and only focus on item \ref{item:secondhomotopysubsolc=0}. 

From item \ref{item:secondhomotopyuniversalupperbound} and the choice of $\mathcal C$ we see that, for any $0\leq \tau \leq 1$,
$$
L_\ep u-cu_s+\mathcal C u\geq 0,\quad u(-a,x)=Kp(x), \quad u(a,x)=0.
$$
Now, let $ \alpha_{\pm}:=\frac{-c\pm\sqrt{c^2+4(1+\varepsilon)\mathcal C}}{2(1+\varepsilon)}
$ and $ m:=K\underset{x\in\mathbb R}{\min}\left(p(x), q(x)\right)>0$.
Then the function $ \theta(s, x)=\theta(s):=m\frac{e^{\alpha_-s+\alpha_+a}-e^{\alpha_+s+\alpha_-a}}{e^{(\alpha_+-\alpha_-)a}-e^{(\alpha_--\alpha_+)a}} $
solves
$$
L_\ep \theta -c\theta _s+\mathcal C \theta =0,\quad \theta(-a)=m, \quad \theta(a)=0.
$$
From the comparison principle, we infer that $u(s,x)\geq \theta (s)$, and similarly $v(s,x)\geq \theta (s)$, for all $(s,x)\in(-a,a)\times \R$. As a result  $ \Normalizationleft{u+v}\geq 2\sup _{(-a_0,a_0)} \theta \geq 2\theta (0)$.

 Next, for $
    c\leq -c^1(a):=-\frac{1+\varepsilon}{a}\ln 4
$
one has
$
    e^{(\alpha_--\alpha_+)a}\leq \frac 14
$
so that
$$ 
    \theta(0)\geq m\frac{e^{\alpha_+ a}-e^{\alpha_- a}}{e^{(\alpha_+-\alpha_-)a}}=m e^{\alpha_-a}\left(1-e^{(\alpha_--\alpha_+)a}\right)\geq m \frac{3e^{\alpha_-a}}{4}.
$$
Next, thanks to a Taylor expansion, we have
$$
    \alpha_-=\frac{-c}{2(1+\varepsilon)}\left(1-\sqrt{1+\frac{4(1+\varepsilon)\mathcal C}{c^2}}\right)=\frac{-c}{2(1+\varepsilon)}\left(-\frac{2(1+\varepsilon)\mathcal C}{c^2}+o\left(\frac{1}{c^2}\right)\right)=\frac{\mathcal C}{c}+o\left(\frac{1}{|c|}\right)
$$
so that there exists $ c^2=c^2(a)>0 $ such that  for any $ c\leq -c^2(a) $ we have $e^{\alpha _-a}>\frac 23$. As a result when $c\leq -
    \underline c(a):=-\max(c^1(a),c^2(a))$, we have
$$
    \Normalizationleft{u+v}\geq  m\geq \nu_0>\nu,
$$
which proves item \ref{item:secondhomotopysubsolc=0}.
\end{proof}

\subsection{Proof of Theorem \ref{thm:existencestrip}}\label{ss:proofexistencestrip}

Equipped with the above estimates, we are now in the position to prove Theorem \ref{thm:existencestrip} using three homotopies and the Leray Schauder topological degree. To do so, let us define the following open subset of $\R\times \Homotopyspace$
\begin{equation*}
    \Gamma:=\left\{\left(c, \vectorize{u\\v}\right)\in \mathbb R\times \Homotopyspace: c\in (0, \bar c^\varepsilon+\ep), \vectorize{0\\0}<\vectorize{u\\v}<\vectorize{C\\C} \text{ in } \Omega\right\}
\end{equation*}
where $\Omega=(-a,a)\times \R$, and $ C >0$ is the  constant defined in Lemma
\ref{lem:firsthomotopyaprioriestimates} item \ref{item:firsthomotopyuniversalupperbound}.

$\bullet$  We develop the first homotopy argument. For $0\leq \tau \leq 1$, let us define the  operator
$$
\begin{array}{rccl} F_\tau: & \mathbb R\times\Homotopyspace & \rightarrow & \mathbb R\times\Homotopyspace \\
 \end{array}
$$
where $ F_{\tau}\left(c, \vectorize{u\\v}\right)=\left(\tilde c, \vectorize{\tilde u\\ \tilde v}\right)$, with 
$$
\tilde c = c + \Normalizationleft{\tilde u+\tilde v}-\nu
$$
and $ \vectorize{\tilde u\\ \tilde v} $ is the unique solution in $ \Homotopyspace $ of the linear problem
$$
\left\{\begin{array}{l}\begin{array}{rcl}
L_\varepsilon \tilde u-c\tilde u_s&=&u(r_u-\gamma_u(u+(\tau v+(1-\tau)\frac qK)))+\mu v-\mu u \\
L_\varepsilon \tilde v-c\tilde v_s&=&v(r_v-\gamma_v((\tau u+(1-\tau)\frac pK)+v))+\mu u-\mu v 
\end{array} \\
(u, v)(-a,x )=(Kp(x), Kq(x)), \quad \forall x\in \R \vspace{3pt}\\
 (u, v)(a, x)=(0, 0),\quad \forall x\in \R.
\end{array}\right.
$$
From standard elliptic estimates, for any $ 0\leq \tau\leq 1$, $ F_\tau $ maps $ \Homotopyspace $ into $ \mathbf C^2_{per}(\overline \Omega) $, which 
shows that $ F_\tau $ is a compact operator in $ \Homotopyspace $. Moreover $F_\tau$ depends continuously on the parameter $0\leq \tau \leq 1$. The Leray-Schauder topological argument can thus be applied: in order to prove that the degree is independent of the parameter $ \tau $,  it suffices to show that there is no fixed point of $ F_\tau $ on the boundary $ \partial\Gamma $, which will be a consequence of estimates in subsection \ref{ss:firsthomotop}. Indeed, let $\left(c,\vectorize{u\\ v}\right)=(c,u,v)$ be a fixed point of $F_\tau$ in $ \overline\Gamma $.
\begin{enumerate}
\item From Lemma \ref{lem:firsthomotopyaprioriestimates}, Lemma \ref{lem:lepestimate} and Remark \ref{rem:consis} we know that if $ c=0 $ then $ \Normalizationleft{u+v}>\nu $ so that $ \tilde c>c $, which is absurd. That shows $ c\neq 0$ .
\item From Lemma \ref{lem:firsthomotopyaprioriestimates} we know that if $ c\geq \bar c^\varepsilon  $ then $ \Normalizationleft{u+v}<\nu $ so that $ \tilde c<c $, which is absurd. That shows $ c<\bar c^\ep +\ep$.
\item From Lemma \ref{lem:firsthomotopyaprioriestimates} we know that $ u<C $ and $ v<C $.
\item From Lemma \ref{lem:firsthomotopyaprioriestimates} and the boundary condition at $ s=-a $, we know that $ u>0 $ and $v>0$ in $ [-a, a)\times \mathbb R $. Moreover, we know from Hopf lemma that $\forall x\in\mathbb R $, $ u_s(a, x)< 0 $ and $ v_s(a, x)< 0 $.
\end{enumerate}
As a result, $ (c, u, v)\notin\partial\Gamma $  so that
\begin{equation}\label{eq:firsthomotopydegreeconservation}
\deg(Id-F_1, \Gamma, 0)=\deg(Id-F_0, \Gamma, 0).
\end{equation}

$\bullet$ We now consider the second homotopy. For $0\leq \tau \leq 1$, let us define the operator
$$
\begin{array}{rccl} G_\tau: & \mathbb R\times\Homotopyspace & \rightarrow & \mathbb R\times\Homotopyspace \\
                            & \vectorize{c, \vectorize{u\\ v}}    & \mapsto     & \vectorize{\tilde c, \vectorize{\tilde u\\ \tilde v}}\end{array}
$$
with again
$$
\tilde c = c +\Normalizationleft{\tilde u +\tilde v}-\nu
$$
and $ \vectorize{\tilde u\\ \tilde v}$ is the unique solutions in $ \Homotopyspace $ of the linear problem
$$
\left\{\begin{array}{l}\begin{array}{rcl}
L_\varepsilon \tilde u-c\tilde u_s + (1-\tau)\mathcal C\tilde u&=&\tau\left(u\left(r_u-\gamma_u\frac qK-\mu -\gamma_u u \right)+\mu  v\right) \\
L_\varepsilon \tilde v-c\tilde v_s + (1-\tau)\mathcal C\tilde v&=&\tau\left(v\left(r_v-\gamma_v\frac pK-\mu -\gamma_v v \right)+\mu  u\right)
\end{array} \\
(u, v)(-a,x )=(Kp(x), Kq(x)), \quad \forall x\in \R \vspace{3pt}\\
 (u, v)(a, x)=(0, 0),\quad \forall x\in \R,  
\end{array}\right.
$$
and $ \mathcal C $ is defined by \eqref{eq:secondhomotopylinearpenality}. Notice that $ G_\tau $ is 
a continuous family of compact operators and that $ G_1=F_0 $. From Lemma \ref{lem:firsthomotopyaprioriestimates} and Lemma \ref{lem:firsthomotopytauequalzero}, we see that there is no fixed point of $F_0$ such that $c\leq 0$ since $ c\mapsto \vectorize{u\\v} $ is nonincreasing. As a result enlarging $\Gamma$ into
\begin{equation*}
    \tilde \Gamma:=\left\{\left(c, \vectorize{u\\v}\right)\in \mathbb R\times \Homotopyspace: c\in (-\underline c(a), \bar c^\varepsilon+\ep), \vectorize{0\\0}<\vectorize{u\\v}<\vectorize{C\\C} \text{ in } \Omega\right\},
   \end{equation*}
with $\underline c(a)\geq 0$ as in Lemma \ref{lem:secondhomotopyaprioriestimates}, does not alter the degree, that is
\begin{equation}\label{eq:firsthomotopytauequalzerochangeopenset}
    \deg (Id-F_0, \Gamma, 0) = \deg(Id-F_0, \tilde \Gamma, 0)=\deg(Id-G_1, \tilde \Gamma, 0).
\end{equation}
Next, using the estimates of Lemma \ref{lem:secondhomotopyaprioriestimates} and Hopf lemma as above, we see that there is no fixed point of $ G_\tau $ on the boundary $ \partial \tilde \Gamma $. We have then
\begin{equation}\label{eq:secondhomotopydegreeconservation}
    \deg(Id-G_1, \tilde \Gamma, 0)=\deg(Id-G_0, \tilde \Gamma, 0).
\end{equation}
Now $ G_0 $ is independent of $ (u, v) $. Since $ L_\varepsilon-c\partial_s +\mathcal CId$ is invertible for each 
$ c\in \mathbb R $, there 
exists exactly one solution of \eqref{eq:secondhomotopy} with $ \tau=0 $ for each $ c\in\mathbb R $, which we denote $ (u_c, v_c) $. Thanks to a 
sliding argument, which we omit 
here, the solutions to  \eqref{eq:secondhomotopy} with $ \tau=0 $ are nonincreasing in $s$ and $ c\mapsto (u_c, v_c) $ is decreasing, so that 
there exists a 
unique $ c\in(-\underline c(a), \bar c ^{\ep}+\ep) $, which we denote $ c_0 $, such that $ (c_0,u_{c_0},v_{c_0}) $ is a fixed point to $ G_0 $.

$\bullet$ Finally a third homotopy allows us to compute the degree.  For $0\leq \tau \leq 1$, let us define the operator $H_\tau:  \mathbb R\times\Homotopyspace  \rightarrow  \mathbb R\times\Homotopyspace 
$ by

$$
    H_\tau(c, u, v)=\left(c+\Normalizationleft{u_c+v_c}-\nu, \tau u_c + (1-\tau) u_{c_0}, \tau v_c+ (1-\tau) v_{c_0}\right).
$$
Noticing that $H_1=G_0$ and that, again, $ H_\tau $ has no fixed point on the boundary $ \partial \tilde \Gamma $, we obtain
\begin{equation}\label{eq:thirdhomotopydegreeconservation}
 \deg(Id-G_0, \tilde \Gamma, 0)=    \deg(Id-H_1, \tilde \Gamma, 0)=\deg(Id-H_0, \tilde\Gamma, 0).
\end{equation}
Then since $ H_0 $ has separated variables and $ c\mapsto\Normalizationleft{u_c+v_c} $ is decreasing, we see that
\begin{equation}\label{eq:thirdhomotopydegreecomputation}
    \deg(Id-H_0, \tilde \Gamma, 0)=1.
\end{equation}

$\bullet$ Combining \eqref{eq:firsthomotopydegreeconservation}, \eqref{eq:firsthomotopytauequalzerochangeopenset}, 
\eqref{eq:secondhomotopydegreeconservation}, \eqref{eq:thirdhomotopydegreeconservation} and \eqref{eq:thirdhomotopydegreecomputation}, we get $\deg(Id-F_1, \Gamma, 0)=1$, 
which shows the existence of a solution to \eqref{eq:pbexistencestrip} in $ \mathbf C^1_{per}(\Omega) $. Theorem \ref{thm:existencestrip} is proved.\qed

\section{Pulsating fronts}\label{s:fronts}

From the previous section,  we are  equipped with  a solution to \eqref{eq:pbexistencestrip} in the strip $(-a,a)\times \R$. From the estimates of Theorem \ref{thm:existencestrip} and standard elliptic estimates, we can --- up to a subsequence--- let $a\to\infty$ and then recover, for any $0<\ep<1$, a speed $0<c=c^\ep<\bar c ^\ep +\ep$ and smooth profiles $(0,0)<(u(s,x),v(s,x))=(u^{\ep}(s,x),v^{\ep}(s,x))< (C,C)$ solving
\begin{equation}\label{pb:towards-epsilon-0}
\left\{
\begin{array}{l}
\begin{array}{rcl}
-u_{xx}-2u_{xs}-(1+\varepsilon)u_{ss}-cu_{s}&=&u(r_u-\gamma_u(u+v))+\mu v-\mu u \quad \text{ in }\R ^{2}\\
-v_{xx}-2v_{xs}-(1+\varepsilon)v_{ss}-cv_{s}&=&v(r_v-\gamma_v(u+v))+\mu u-\mu v  \quad \text{ in }\R ^{2}
\end{array} \\
(u, v)(s, \cdot) \quad \text{ is $L$-periodic} \\
\Normalizationleft{u+v} = \nu.
\end{array}\right.
\end{equation}
Let us mention again that, because of the lack of comparison, we do not know that the above solution is decreasing in $s$, in sharp contrast with the previous results on pulsating fronts \cite{Wei-02}, \cite{Ber-Ham-02}, \cite{HZ},
\cite{Ber-Ham-Roq-05-n2}, \cite{Ham-08}, \cite{Ham-Roq-11}. To overcome this lack of monotony, further estimates will be required. 

Now, the main difficulty is to show that, letting $\ep \to 0$, we recover a nonzero speed and thus a pulsating front. 
To do so, it is not convenient to use the $(s,x)$ variables, and we therefore switch to functions
$$
\tilde u(t, x):=u(x-ct, x),\quad
\tilde v(t, x):=v(x-ct, x), \quad (t,x)\in \R ^{2},
$$
which are consistent with Definition \ref{def:pul} of a pulsating front. Hence, after dropping  the tildes, \eqref{pb:towards-epsilon-0} is recast
\begin{equation}\label{pb:towards-epsilon-0-parabolic}
\left\{\begin{array}{l}
\begin{array}{rcl}
-\frac{\varepsilon}{c^2} u_{tt}- u_{xx}+ u_t&=& u(r_u-\gamma_u( u+ v))+\mu  v-\mu  u \quad \text{ in }\R ^{2}\vspace{3 pt}\\
-\frac{\varepsilon}{c^2} v_{tt}- v_{xx}+ v_t&=& v(r_v-\gamma_v(u+ v))+\mu  u-\mu  v \quad \text{ in }\R ^{2}\vspace{3 pt}\\
\end{array} \\
\underset{x-ct\in (-a_0, a_0)}{\sup}\, u(t, x)+ v(t, x) = \nu.
\end{array}\right.
\end{equation}
Also the $L$ periodicity for \eqref{pb:towards-epsilon-0} is transferred into the constraint \eqref{eq:proppuls} for  \eqref{pb:towards-epsilon-0-parabolic}. Moreover, up to a translation, we can assume w.l.o.g. that the solution to \eqref{pb:towards-epsilon-0-parabolic} satisfies
\begin{equation}\label{eq:realnorm}
    \underset{x\in (-a_0, a_0)}{\sup}( u(0, x)+ v(0, x))=\nu.
\end{equation}
Also,  though $ t $ can be interpreted as a time, we would like to stress out that \eqref{pb:towards-epsilon-0-parabolic} is not a Cauchy problem.

Our first goal in this section  is to let $\ep \to 0$ in \eqref{pb:towards-epsilon-0-parabolic} and get the following.

\begin{thm}[Letting the regularization tend to zero]\label{thm:existencepuls}
There exist a speed $0<c\leq \bar c ^{0}:=\lim _{\ep \to 0} \bar c ^{\ep}$ (see Lemma \ref{lem:speed}) and  positive profiles $ (u,v) $ solving, in the classical sense,
\begin{equation}\label{eq:pb-parabolic-line}
\left\{\begin{array}{rcl}
u_t-u_{xx}&=&u(r_u-\gamma_u(u+v))+\mu(v-u) \quad \text{ in } \R^{2}\\
v_t-v_{xx}&=&v(r_v-\gamma_v(u+v))+\mu(u-v) \quad \text{ in } \R^{2},
\end{array}\right.
\end{equation}
satisfying the constraint \eqref{eq:proppuls} and, for some $a_0>0$, the normalization
\begin{equation*}
\pulsenorm{u+v}=\nu.
\end{equation*}
\end{thm}

The present section is organized as follows. After proving further estimates on solutions to \eqref{pb:towards-epsilon-0-parabolic} in subsection \ref{ss:lowerest}, we prove Theorem \ref{thm:existencepuls} in subsection \ref{ss:exist}, the main difficulty being to exclude the possibility of a standing wave. Finally, in subsection \ref{ss:proofpuls} we conclude the construction of a pulsating front, thus proving our main result Theorem \ref{th:pulsating}.

\subsection{Lower estimates on solutions to \eqref{pb:towards-epsilon-0-parabolic}} \label{ss:lowerest}

We start by showing a uniform lower bound on the solutions to \eqref{pb:towards-epsilon-0-parabolic} that have a positive lower bound. The argument relies on the sign of the eigenvalue $\l$, or more precisely that of the first eigenvalue to the stationary Dirichlet problem in large bouded domains. For $b>0$, we denote 
$ (\lambda_1^b, \Phi^b) $ with $ \Phi^b(x):=\vectorize{\varphi^b (x) \\ \psi^b(x)} $ the unique eigenpair solving
\begin{equation}\label{eq:stationaryDirichlet}
    \left\{\begin{array}{l}
    -\Phi^b_{xx}-A(x)\Phi^b = \lambda_1^b\Phi^b \\
    \varphi^b(x)>0,\; \psi^b(x)>0, \quad x\in (-b, b) \\
    \varphi^b(\pm b)=\psi^b(\pm b)=0,
\end{array}
\right.
\end{equation}
and $ \Vert\Phi^b\Vert_{\Lp\infty(-b, b)}=1$. From Lemma \ref{lem:annexeC}, we know that $ \lambda_1^b \to \l <0$ when $ b\to\infty $. We can thus select $a_1>a_0^*$, with $a_0^*$ as in Theorem \ref{thm:existencestrip}, large enough so that
\begin{equation}\label{eq:defa1}
b\geq a_1 \Rightarrow \lambda_1^b\leq\frac{3\l}{4}.
\end{equation}
Also, from  Hopf lemma we have  $ C^b:=\underset{x\in(-b, b)}\sup\left(\frac{\varphi^b(x)}{\psi^b(x)}, \frac{\psi^b(x)}{\varphi^b(x)}\right) <+\infty$.

\begin{lem}[A uniform lower estimate]\label{lem:infpositiveinfnotsosmall}
    Let $ (u(t,x),v(t,x)) $ be a classical positive solution to 
\begin{equation}\label{blabla}
    \left\{
        \begin{array}{rcl}
            \beta u_t-\kappa u_{tt}-u_{xx}&=&u(r_u-\gamma_u(u+v))+\mu v-\mu u \quad \text{ in } \R^2\\
            \beta v_t-\kappa v_{tt}-v_{xx}&=&v(r_v-\gamma_v(u+v))+\mu u-\mu v  \quad \text{ in } \R^2,
        \end{array} 
    \right.
\end{equation}
with $ \kappa\geq 0 $ and $ \beta\in\mathbb R $. Let also $ b\geq a_1 $ and $ \Phi^b $ the solution to \eqref{eq:stationaryDirichlet}.

Then there exists a constant $ \alpha_0=\alpha_0(\minmu, \maxgamma, \lambda_1^b, C^b) >0$ such that if
$$
\underset{(t, x)\in\mathbb R\times (-b, b)}\inf\min(u(t,x),v(t,x))>0
$$
then 
$$
\forall (t, x)\in \mathbb R\times (-b, b),\; \vectorize{u(t, x)\\v(t,x)}\geq\alpha_0\Phi^b(x).
$$
\end{lem}

\begin{proof} Let $0<\eta \leq 1$ be given. For $\alpha >0$, we define
$$
    \vectorize{ U^{\alpha, \eta}(t, x) \\ V^{\alpha, \eta}(t, x)}:=\alpha(1-\eta t^2)\vectorize{\varphi^b(x)\\{\psi^b}(x)}.
$$
Then for small $ \alpha<\min\left(\underset{(t, x)\in\mathbb R\times (-b, b)}\inf u, \underset{(t, x)\in\mathbb R\times (-b, b)}\inf v\right) $ we have 
$ \vectorize{ U^{\alpha, \eta}(t, x) \\ V^{\alpha, \eta}(t, x)}\leq \vectorize{ u(t, x) \\ v(t, x)} $ for all $(t, x)\in\mathbb R\times (-b, b)$,  whereas for large $ \alpha>\frac{\max(u(0, 0), v(0, 0))}{\min(\varphi^b(0), \psi^b(0))} $ one has 
$ \vectorize{ U^{\alpha, \eta}(0,0) \\ V^{\alpha, \eta}(0,0)}>\vectorize{ u(0,0) \\ v(0, 0)} $. 
Thus we can define
$$ 
    \alpha_0^\eta=\alpha _0:=\sup\left\{\alpha>0, \forall (t, x)\in\mathbb R\times (-b, b), \vectorize{ U^{\alpha, \eta}(t, x) \\ V^{\alpha, \eta}(t, x)}\leq \vectorize{ u(t, x) \\ v(t, x)}\right\}>0.
$$

 Assume by contradiction that
$$
\alpha_0\leq\alpha_0^*:=\min\left(1, \frac{\minmu}{2\maxgamma},
    \frac{-\lambda_1^b}{2(1+2C^b)\maxgamma}\right).
$$
There exists a touching point $ (t_0, x_0) \in (-\sqrt{\eta}, \sqrt \eta)\times (-b, b) $ such that either 
$ u(t_0, x_0)=U^{\alpha_0, \eta}(t_0, x_0) $ or $ v(t_0, x_0)=V^{\alpha_0, \eta}(t_0, x_0) $. Assume $ u(t_0, x_0)=U^{\alpha_0, \eta}(t_0, x_0) $ for instance. Then $u-U^{\alpha_0, \eta}$ reaches a zero minimum at $(t_0,x_0)$ so that
\begin{eqnarray*}
    0&\geq& \beta\left(u-U^{\alpha_0, \eta}\right)_t-\kappa\left(u-U^{\alpha_0, \eta}\right)_{tt}-\left(u-U^{\alpha_0, \eta}\right)_{xx} \\
    &=&(\beta u_t-\kappa u_{tt}-u_{xx}) + \alpha_0(1-\eta t_0^2)\varphi^b_{xx}+2\alpha_0\beta\eta t_0\varphi^b-2\alpha_0\kappa\eta\varphi^b
\end{eqnarray*}
at point $(t_0,x_0)$. Using \eqref{eq:stationaryDirichlet} and \eqref{blabla} yields
$$
    0\geq u(r_u-\mu-\gamma_u(u+v))+\mu v - \alpha_0(1-\eta t_0^2)(\varphi^b(r_u-\mu+\lambda_1^b)+\mu{\psi^b}) +2\alpha_0\eta \varphi^b(\beta t_0-\kappa)
$$
at point $(t_0,x_0)$, and since $ u(t_0, x_0)=\alpha_0(1-\eta t_0^2)\varphi^b(x_0)$ we end up with
\begin{equation}\label{twocases}
    0\geq u_0[-\lambda_1^b-\gamma_u(x_0)(u_0+v_0)]+\mu(x_0) [v_0-\alpha_0(1-\eta t_0^2){\psi^b}(x_0)] +2\alpha_0\eta \varphi^b(x_0)(\beta t_0-\kappa),
\end{equation}
with the notations $u_0=u(t_0,x_0)$, $v_0=v(t_0,x_0)$. Now two cases may occur.

$\bullet$
    Assume first that $ v_0\leq 2\alpha_0(1-\eta t_0^2) {\psi^b}(x_0) $. Then we have 
$$ 
    v_0\leq  2\alpha_0(1-\eta t_0^2) \frac{{\psi^b}(x_0)}{\varphi^b(x_0)}\varphi^b(x_0) \leq  2C^bu_0,
$$
and since $ v_0-\alpha_0(1-\eta t_0^2){\psi^b}(x_0)\geq 0 $, we deduce from \eqref{twocases} that
$$
    \gamma_u(x_0)(1+2C^b)u^2_0\geq -\lambda_1^bu_0 + 2\alpha_0\eta \varphi^b(x_0)(\beta t_0-\kappa),
$$
which in turn implies
$$
    \gamma ^\infty(1+2C^b)\alpha_0\geq\gamma_u(x_0)(1+2C^b)u_0\geq -\lambda_1^b + \frac{2\alpha_0\eta \varphi^b(x_0)(\beta t_0-\kappa)}{ u_0}\geq -\lambda_1^b - \frac{2\eta}{\inf u}(|\beta||t_0|+\kappa),
$$
since $\alpha _0\leq 1$ and $\varphi^b\leq 1$. Since $ |t_0|\leq \frac{1}{\sqrt\eta} $, one then has
\begin{equation}\label{un}
    \alpha_0\geq \frac{-\lambda_1^b}{(1+2C^b)\maxgamma}-2\sqrt\eta\frac{|\beta|+\kappa}{(1+2C^b)\maxgamma\inf u}.
\end{equation}

$\bullet$ 
On the other hand, assume $ v_0\geq 2\alpha_0(1-\eta t_0^2) {\psi^b}(x_0) $. Then we deduce from \eqref{twocases} that
\begin{eqnarray*}
    \gamma_u (x_0)u^2_0 &\geq &-\lambda_1^bu_0 + \frac{\mu(x_0)}{2} (v_0-2\alpha_0(1-\eta t_0^2){\psi^b}(x_0))+v_0\left(\frac{\mu(x_0)}{2}-\gamma_u(x_0) u_0\right) \\
    &&+2\alpha_0\eta \varphi^b(x_0)(\beta t_0-\kappa)\\
&\geq& -\lambda_1^bu_0+ 2\alpha_0\eta \varphi^b(x_0)(\beta t_0-\kappa),
\end{eqnarray*}   
since  $ \gamma_uu\leq\gamma_u\alpha_0^*\leq \frac{\minmu}{2}$. Arguing as in the first case, we end up with
\begin{equation}\label{deux}
    \alpha_0\geq\frac{-\lambda_1^b}{\maxgamma}-2\sqrt\eta\frac{|\beta|+\kappa}{\maxgamma\inf u}.
\end{equation}

From \eqref{un} , \eqref{deux} and the symmetric situation where $v(t_0, x_0)=V^{\alpha_0, \eta}(t_0, x_0) $, we deduce that, in any case,
\begin{equation}\label{eq:proofleminfisnotsmall}
    \alpha_0\geq\frac{-\lambda_1^b}{(1+2C^b)\maxgamma}-2\sqrt\eta\frac{|\beta|+\kappa}{\maxgamma\inf(u,v)}.
\end{equation}
One sees that for 
$$
    0<\eta<\eta^*:=\min\left(1, \left(\frac{-\lambda_1^b\inf(u,v)}{4(|\beta|+\kappa)(1+2C^b)} \right)^2\right),
$$
inequality \eqref{eq:proofleminfisnotsmall} is a contradiction since it implies $ \alpha_0> \alpha_0^* $. Hence we have shown that for any $ 0<\eta<\eta^* $ one has 
$ \alpha_0=\alpha_0^\eta>\alpha_0^* $. In particular 
$$
    \forall \eta\in(0, \eta^*),\forall (t, x)\in\mathbb R\times (-b, b),  \vectorize{u(t, x)\\v(t, x)} \geq \alpha_0^*(1-\eta t^2)\vectorize{\varphi^b(x) \\ {\psi^b}(x)}.
$$
Taking the limit $ \eta\to 0 $, we then obtain 
$$
\forall (t, x)\in\mathbb R\times (-b, b), \vectorize{u(t, x) \\ v(t, x)}\geq\alpha_0^*\Phi^b(x),
$$
which concludes the proof of  Lemma \ref{lem:infpositiveinfnotsosmall}.
\end{proof}

Next we establish a forward-in-time lower estimate for solutions of the (possibly degenerate) problem \eqref{blabla2}. 
The proof is based on the same idea as in Lemma \ref{lem:infpositiveinfnotsosmall}, but it is here critical that the coefficient $\beta$ of the time-derivative has the right sign. Roughly speaking, the following lemma asserts that once a population has reached a certain threshold on a large enough set, it cannot fall under that threshold at a later time. 

\begin{lem}[A forward-in-time lower estimate]\label{lem:forward}  Let $ (u(t,x),v(t,x)) $ be a classical positive solution to 
\begin{equation}\label{blabla2}
    \left\{
        \begin{array}{rcl}
            \beta u_t-\kappa u_{tt}-u_{xx}&=&u(r_u-\gamma_u(u+v))+\mu v-\mu u \quad \text{ in } \R^2\\
            \beta v_t-\kappa v_{tt}-v_{xx}&=&v(r_v-\gamma_v(u+v))+\mu u-\mu v  \quad \text{ in } \R^2,
        \end{array} 
    \right.
\end{equation}
with $ \kappa\geq 0 $ and $ \beta\geq 0 $. Let also $ b\geq a_1 $ and $ \Phi^b $ the solution to \eqref{eq:stationaryDirichlet}.

Then there exists a constant $ \alpha_0=\alpha_0(\minmu, \maxgamma, \lambda_1^b, C^b) >0$ such that if
 $ 0 < \alpha < \alpha_0 $ and 
\begin{equation}\label{tzero}
\forall x\in(-b,b),\; \alpha\Phi^b(x)<\vectorize{u(0, x) \\ v(0, x)},
 \end{equation}
 then 
$$
\forall t>0,\forall x\in(-b, b),\; \alpha\Phi^b(x)\leq\vectorize{u(t, x) \\ v(t, x)}.
$$
\end{lem}
\begin{proof}
Let
$$
 0<\alpha<\alpha_0:=\min\left(1, \frac{-\lambda_1^b}{2(1+2C^b)\maxgamma}, \frac {\minmu}{2\maxgamma}\right) 
$$
and assume \eqref{tzero}. For $\eta >0$ we define
$$
   \zeta(t, x)= \left(\begin{matrix} \zeta_u(t,x) \\ \zeta_v(t,x)\end{matrix}\right):=\alpha(1-\eta t)\left(\begin{matrix}\varphi^b(x)\\ \psi^b(x)\end{matrix}\right).
$$
From \eqref{tzero}, we can define
$$
	\eta_0:=\inf\left\{\eta \in \R: \forall t\geq 0, \forall x\in [-b, b], \left(\begin{matrix} u(t,x) \\ v(t,x)\end{matrix}\right)\geq\zeta (t,x) \right\}.
$$
Assume by contradiction that $ \eta_0>0 $. Then there exists $ t_0>0 $ and $ x_0\in(-b, b) $ such that, say,  $ u(t_0, x_0)=\zeta_u(t_0, x_0) $.  
Then at point $ (t_0, x_0)$  we have
$$
0\geq \beta (u-\zeta_u)_t-\kappa (u-\zeta_u)_{tt} -(u-\zeta_u)_{xx}=u(r_u-\gamma_u(u+v))+\mu(v-u)+{\zeta_u}_{xx}+\beta\alpha\eta\varphi^b.
$$
Using \eqref{eq:stationaryDirichlet} and $u(t_0,x_0)=\alpha (1-\eta _0t_0)\varphi ^b(x_0)$, we end up with
\begin{equation}\label{plusgrandzero}
	0\geq u_0(-\lambda_1^b-\gamma_u(x_0)(u_0+v_0))+\mu(x_0)(v_0-\zeta_v(t_0,x_0)),
\end{equation}
with the notations $u_0=u(t_0,x_0)$, $v_0=v(t_0,x_0)$ and thanks to $\beta \geq 0$.
Now two cases may occur.

$\bullet$ Assume first  that $ v_0\leq 2\zeta_v(t_0,x_0)$. Then $ v_0\leq 2\frac{\zeta_v(t_0,x_0)}{\zeta_u(t_0,x_0)}\zeta_u(t_0,x_0)\leq 2C^b\zeta_u(t_0,x_0)=2C^bu_0 $, so that \eqref{plusgrandzero} yields (recall that $v_0\geq \zeta _v (t_0,x_0)$)
$$
	\gamma_u(x_0)(1+2C^b)u^2_0\geq \gamma_u(x_0)(u_0+v_0)u_0\geq -\lambda_1^bu_0.
$$
As a result $u_0>\alpha_0$, which is  a contradiction. 

$\bullet$  Assume now that $ v_0 \geq 2\zeta_v(t_0,x_0) $. Then we deduce from \eqref{plusgrandzero} that
\begin{eqnarray*}
    \gamma_u(x_0)u^2_0&\geq& -\lambda_1^bu_0+v_0\left(\frac{\mu(x_0)}{2}-\gamma_u(x_0) u_0\right) +\frac{\mu(x_0)}{2}(v_0-2\zeta_v (t_0,x_0))\\
	&\geq & -\lambda_1^bu_0+\frac 12\mu(x_0)(v_0-2\zeta_v(t_0,x_0)),
\end{eqnarray*}
since $u_0\leq \alpha _0\leq \frac{\mu ^0}{2\gamma ^\infty}$. As a result $ u_0\geq \frac{-\lambda_1^b}{\gamma ^\infty}>\alpha _0 $, which is also a  contradiction.

 Thus $ \eta_0\leq 0 $ and in particular
$$
	\forall t>0,\forall x\in(-b, b),\; \left(\begin{matrix}u(t, x) \\ v(t, x)\end{matrix}\right)\geq \alpha\left(\begin{matrix} \varphi^b(x) \\ \psi^b(x) \end{matrix}\right),
$$
which concludes the proof of Lemma \ref{lem:forward}.
\end{proof}

\subsection{Proof of Theorem \ref{thm:existencepuls}}\label{ss:exist}

In this subsection, we prove that a well-chosen series of solutions to equation \eqref{pb:towards-epsilon-0-parabolic} cannot converge, as $\ep \to 0$, to a standing wave ($c=0$). In other words, we prove Theorem \ref{thm:existencepuls}, making a straightforward use of the crucial Lemma \ref{lem:cnotto0}. The rough idea of the proof of Lemma \ref{lem:cnotto0} is that a standing wave cannot stay in the neighborhood of 0 for a long time. Hence the normalization allows us to prevent a sequence of 
solutions from converging to a standing wave, provided $\nu$ is chosen small enough. Notice also that the interior gradient estimate for elliptic systems of Lemma \ref{lem:bernsteinestimate} will be used.

In the sequel we select $ a_1>a_0^* $ as in  \eqref{eq:defa1}, recall that $ \lambda_1^{a_1} $ denotes the eigenvalue of problem \eqref{eq:stationaryDirichlet} in the
domain $ (-a_1, a_1) $, and  define 
$$
    \nu^* := \frac 12\min \left(\nu_0, \underline \nu\right)>0,
$$
where $ \underline \nu:=\alpha_0\underset{x\in(-a_0^*, a_0^*)}\inf\min(\varphi^{a_1}(x), \psi^{a_1}(x)) $, with $ \alpha_0>0 $ the constant in Lemma \ref{lem:infpositiveinfnotsosmall} in the domain $ (-a_1, a_1) $.

\begin{lem}[Nonzero limit speed]\label{lem:cnotto0}
Let $ (\varepsilon_n, c_n, u^n(t,x), v^n(t,x)) $ be a sequence such that $\varepsilon_n>0 $, $ \varepsilon_n\to 0$,  $c_n\neq 0$, $ (u^n, v^n) $ is a positive solution to problem \eqref{pb:towards-epsilon-0-parabolic} with $ \varepsilon=\varepsilon_n$, $c=c_n$, $ 0<\nu< \nu^* $ and $ a_0>a_1$. Then 
\begin{equation}
\underset{n\to\infty}{\liminf}\,c_n>0.
\end{equation}
\end{lem}
\begin{proof}
Assume by contradiction that there is a sequence as in Lemma \ref{lem:cnotto0} with 
$\lim\,c_n=0$. Define the sequence $ \kappa_n := \frac{\varepsilon_n}{c_n^2} >0$ which, up to an extraction,  tends  to $ +\infty $, or to some $ \kappa\in(0, +\infty) $ or to $ 0 $. In each case we 
are going to construct a couple of functions $ (u,v) $ that shows a contradiction. We refer to  \cite{Ber-Ham-02} or to
\cite{Ber-Ham-Roq-05-n2} for a similar trichotomy.
\medskip

\noindent \textbf{Case 1: $\kappa_n\to +\infty $.} Defining  $ (\tilde u^n, \tilde v^n)(t,x):=(u^n, v^n)(\sqrt{\kappa _n}t,x) $, problem \eqref{pb:towards-epsilon-0-parabolic} is recast
\begin{equation}\label{kappainfini}
\left\{\begin{array}{l}
\begin{array}{rcl}
-u^n_{tt}-u^n_{xx}+\frac{1}{\sqrt{\kappa_n}}u^n_{t}&=&u^n(r_u-\gamma_u(u^n+v^n))+\mu v^n-\mu u^n \\
-v^n_{tt}-v^n_{xx}+\frac{1}{\sqrt{\kappa_n}}v^n_{t}&=&v^n(r_v-\gamma_v(u^n+v^n))+\mu u^n-\mu v^n \\
\end{array} \\
\underset{x-{\sqrt{\varepsilon_n}}t\in (-a_0, a_0)}{\sup}\,u^n(t,x)+v^n(t,x) = \nu,
\end{array}\right.
\end{equation}
where we have dropped the tildes. From  standard elliptic estimates, this sequence converges, up to an extraction, to a classical nonnegative solution $ (u,v) $ of
\begin{equation} \label{eq:limitctozerobigcrespepsilon}
    \left\{
        \begin{array}{rcl}
            -u_{tt}-u_{xx}&=&u(r_u-\gamma_u(u+v))+\mu v-\mu u \\
            -v_{tt}-v_{xx}&=&v(r_v-\gamma_v(u+v))+\mu u-\mu v, \\
        \end{array}
    \right.
\end{equation}
and since $ (u^n, v^n) $ satisfies the third equality in \eqref{kappainfini} together with \eqref{eq:realnorm}, $ (u,v) $ satisfies $ \underset{(t, x)\in\mathbb R\times (-a_0, a_0)}{\sup}(u+v)= \nu $. In particular, $(u,v)$ is nontrivial and thus positive by the strong maximum principle.

Now, applying Lemma \ref{lem:forward} to $ (u, v) $ with $ \alpha:=\frac 12\min\left(\underset{x\in(-a_0, a_0)}\inf(u(0, x), v(0,x)), \alpha_0\right)>0$, we get
$$
\forall t>0, \forall x\in (-a_0, a_0), \vectorize{u(t, x) \\ v(t, x)}\geq\alpha\Phi^{a_0}(x).
$$
Next, thanks to standard elliptic estimates, the sequence
$$
 (u^n(t, x), v^n(t, x)):=(u(t+n, x), v(t+n, x))
 $$
  converges, up to an extraction, to a solution $ (u,v) $ of
\eqref{eq:limitctozerobigcrespepsilon} --- that we denote again by $(u,v)$--- which satisfies
\begin{equation}\label{eq1}
\underset{(t, x)\in\mathbb R\times (-a_0, a_0)}\sup(u+v)=\nu,
\end{equation}
and
$$    \forall (t,x)\in\mathbb R\times (-a_0, a_0), \vectorize{u(t, x)\\v(t, x)}\geq\alpha\Phi^{a_0}(x). $$
In particular, since $a_0>a_1$, the latter implies
\begin{equation}\label{eq2}
 \underset{(t,x)\in \mathbb R\times (-a_1, a_1)}\inf\min(u,v)>0.
 \end{equation}

\medskip
\noindent \textbf{Case 2: $ \kappa_n\to \kappa\in (0, +\infty) $.} Thanks to  standard elliptic estimates, the sequence $ (u^n, v^n) $ converges, up to an extraction, to a solution $ (u,v) $ of 
\begin{equation}\label{eq:ctozerocase2}
\left\{
    \begin{array}{rcl}
        -\kappa u_{tt}-u_{xx}+u_t&=&u(r_u-\gamma_u(u+v))+\mu v-\mu u \\
        -\kappa v_{tt}-v_{xx}+v_t&=&v(r_v-\gamma_v(u+v))+\mu u-\mu v,
    \end{array}
\right.
\end{equation}
and since $ (u^n, v^n) $ satisfies the third equality in \eqref{pb:towards-epsilon-0-parabolic} together with \eqref{eq:realnorm}, $ (u,v) $ satisfies  $ \underset{(t, x)\in\mathbb R\times (-a_0, a_0)}{\sup}(u+v)=\nu $. In particular,  $(u,v)$ is nontrivial and thus positive by the strong maximum principle.

Now, using Lemma \ref{lem:forward} and a positive large shift in time exactly as in Case 1, we end up with a   solution $ (u,v) $ to
\eqref{eq:ctozerocase2} which satisfies \eqref{eq1} and \eqref{eq2}.

\medskip

\noindent  \textbf{Case 3: $ \kappa_n\to 0 $.} In this case, the elliptic operator becomes degenerate as $n\to\infty$, so that we cannot use the standard elliptic theory. The idea is then to use a Bernstein interior gradient estimate for elliptic systems that we present and prove in Appendix \ref{s:bernstein}.

Applying Lemma \ref{lem:bernsteinestimate} to the series $ (u^n, v^n)$ solving \eqref{pb:towards-epsilon-0-parabolic}, we get a uniform $ L^\infty $ bound for $ (u_x^n, v_x^n)$. 
Furthermore by differentiating \eqref{pb:towards-epsilon-0-parabolic} with respect to $x$, we see that $(u^n_x,v^n_x)$ solves a system for which Lemma \ref{lem:bernsteinestimate} still applies. As a result, we get a uniform $ L^\infty $ bound  for $ (u_{xx}^n ,v^n_{xx}) $.  

Let us show that there is also a uniform $L^\infty$ bound for $ (u^n_t, v^n_t) $.  From the uniform bounds found above, we can write
\begin{equation*}
u^n_t-\kappa_n u^n_{tt} = F^n(t,x). 
\end{equation*}
Let $F:=\max(1,\sup _n \Vert F^n\Vert _{L^\infty(\R ^2)})<+\infty$. Assume by contradiction that there is a point $(t_0,x_0)$ where $ u^n_t(t_0,x_0)> 2F $. From the above equation we deduce that $u^n_t(t,x_0)>2F$ remains valid for $t\geq t_0$, and thus  
\begin{equation*}
\kappa_n u^n_{tt}(t,x_0)> F, \quad \forall t\geq t_0.
\end{equation*}
Integrating twice, we get
\begin{equation*}
u^n(t, x_0)\geq F (2(t-t_0)+\frac{1}{2\kappa_n}(t-t_0)^2) -\Vert u^n\Vert_{L^\infty}, \quad \forall t\geq t_0.
\end{equation*}
Letting $t\to \infty$ we get that $u^n$ is unbounded, a contradiction. Thus, $ u^n_t(t, x)\leq 2F$ for any $ (t, x)\in\mathbb R^2$ and, in a straightforward way, $\vert u^n_t(t,x)\vert, \vert v^n_t(t,x)\vert \leq 2F$ for any $(t,x)\in\R ^2$.

Since we have uniform $L^\infty$ bounds for $ (u^n, v^n) $, $(u^n_x,v^n_x)$ and $(u^n_t,v^n_t)$, there are $u$ and $v$ in $H^1_{loc}(\R ^2)$ such that, up to a subsequence,
$$
(u^n, v^n) \to (u,v) \text{ in  $L^{\infty}_{loc}(\R ^{2})$},\quad 
 (u_x^n, v_x^n,u_t^n, v_t^n)\rightharpoonup (u_x, v_x,u_t, v_t) \text{  in  $L^2_{loc}(\R^2) $ weak}.
$$
As a result, letting $n\to \infty$ into \eqref{pb:towards-epsilon-0-parabolic} yields
\begin{equation} \label{eq:limitctozeroparabolic}
\left\{
\begin{array}{rcl}
u_t-u_{xx}&=&u(r_u-\gamma_u(u+v))+\mu v-\mu u \\
v_t-v_{xx}&=&v(r_v-\gamma_v(u+v))+\mu u-\mu v \\
\end{array} 
\right.
\end{equation}
in a weak sense. From parabolic regularity, $ (u,v) $ is actually a classical solution to \eqref{eq:limitctozeroparabolic}. Since the convergence occurs locally uniformly \eqref{eq:realnorm} and since $ (u^n, v^n) $ satisfies the third equality in \eqref{pb:towards-epsilon-0-parabolic} together with \eqref{eq:realnorm}, $ (u,v) $ satisfies  $ \underset{(t, x)\in\mathbb R\times (-a_0, a_0)}{\sup}(u+v)=\nu $. In particular,  $(u,v)$ is nontrivial and thus positive by the strong maximum principle.

Now, using Lemma \ref{lem:forward} and a positive large shift in time as in Case 1 (parabolic estimates replacing elliptic estimates), we end up with a   solution $ (u,v) $ to
\eqref{eq:limitctozeroparabolic} which satisfies \eqref{eq1} and \eqref{eq2}.

\medskip

 \noindent \textbf{Conclusion.} In any of the three above cases, we have constructed a classical solution $ (u,v) $ to ($\beta\geq 0$, $\kappa \geq 0$)
$$
\left\{
\begin{array}{rcl}
    \beta u_t-\kappa u_{tt}-u_{xx}&=&u(r_u-\gamma_u(u+v))+\mu v-\mu u \\
    \beta v_t-\kappa v_{tt}-v_{xx}&=&v(r_v-\gamma_v(u+v))+\mu u-\mu v ,\\
\end{array} 
\right.
$$
which satisfies \eqref{eq1} and \eqref{eq2}. Applying Lemma \ref{lem:infpositiveinfnotsosmall}, we find that (recall that $a_1>a_0^*$)
$$
\underset{\mathbb R\times(-a_0^*, a_0^*)}\inf(u,v)\geq\alpha_0  \underset{(-a_0^*, a_0^*)}\inf(\varphi^{a_1},\psi^{a_1})=\underline\nu.
$$
But, since $a_0>a_0^{*}$ the above implies
$$
\underset{\mathbb R\times (-a_0, a_0)}\sup(u+v)\geq 2 \underset{\mathbb R\times(-a_0^*, a_0^*)}\inf(u,v)\geq 2\underline \nu>\nu ^{*}>\nu,
$$
which contradicts \eqref{eq1}. Lemma \ref{lem:cnotto0} is proved.\end{proof}

We are now in the position to prove Theorem \ref{thm:existencepuls}.

\begin{proof}[Proof of Theorem \ref{thm:existencepuls}]
From the beginning of Section \ref{s:fronts} and Lemma \ref{lem:cnotto0} we can consider a sequence $ (\varepsilon_n, c_n, u^n(t,x), v^n(t,x)) $  such that $\varepsilon_n>0 $, $ \varepsilon_n\to 0$,  $0<c_n\leq \bar c^{\ep _n}+\ep_n$, $ (u^n, v^n) $ is a positive solution to problem \eqref{pb:towards-epsilon-0-parabolic} with $ \varepsilon=\varepsilon_n$, $c=c_n$, $ \nu< \nu^* $ and $ a_0>a_1$, satisfying the constraint \eqref{eq:proppuls}, and the crucial  fact
\begin{equation}
\underset{n\to\infty}{\lim}\,c_n>0.
\end{equation}
Notice that, as a by-product,  
this shows that  $\bar c^0:= \lim_{\varepsilon\to 0}\bar c^\varepsilon>0 $ (see Lemma \ref{lem:speed}).  We can now repeat the argument in the proof of Lemma \ref{lem:cnotto0} Case 3 and extract a sequence $ (u^n, v^n) $ which converges to a 
classical solution $(u,v)$ of equation 
\eqref{eq:pb-parabolic-line}, satisfying the 
normalization
\begin{equation*}
\pulsenorm{u+v}=\nu
\end{equation*}
as well as the constraint \eqref{eq:proppuls}. Theorem \ref{thm:existencepuls} is proved.
\end{proof}

\subsection{Proof of Theorem \ref{th:pulsating}}\label{ss:proofpuls}

We are now close to conclude the proof of our main result of construction of a pulsating front, Theorem \ref{th:pulsating}. From Theorem \ref{thm:existencepuls}, it only remains to prove the boundary conditions \eqref{weak-boundary}, namely
\begin{equation*}
    \liminf_{t\to+\infty} \vectorize{u(t,x)\\v(t,x)}>\vectorize{0 \\ 0},\quad \lim_{t\to-\infty} \vectorize{u(t,x)\\v(t,x)}=\vectorize{0 \\ 0}, \quad \text{locally uniformly w.r.t. $x$,}
\end{equation*}
to match Definition \ref{def:pul} of a pulsating front. The former is derived by another straighforward application of Lemma \ref{lem:forward}, while the latter is proved below. Hence, Theorem \ref{th:pulsating} is proved. \qed

\begin{lem}[Zero limit behavior]\label{thm:tozero} For $a_1>a_0^*$ and $\nu ^*>0$ as in subsection \ref{ss:exist}, let $c>0$ and $(u,v)$ be as in Theorem \ref{thm:existencepuls}, satisfying in particular  the normalization $ \underset{x-ct\in(-a_0, a_0)}\sup(u+v)=\nu $ with $ \nu<\nu^* $ and $ a_0>a_1 $. Then
$$
	\lim_{t\to -\infty}\max(u,v)(t,x)\to 0, \quad \text{locally uniformly w.r.t. $x$.}
$$
\end{lem}

\begin{proof}  We first claim that $ \underset{\mathbb R\times(-a_0, a_0)}\inf\min\, (u,v)=0 $. Indeed if this is not the case then, in particular,  $ \underset{\mathbb R\times(-a_1, a_1)}\inf\min\, (u,v)>0 $, and we derive a contradiction via  Lemma \ref{lem:infpositiveinfnotsosmall} by a straightforward adaptation of the Conclusion of the proof of Lemma \ref{lem:cnotto0}, because $ \mathbb R\times (-a_1, a_1) $ intersects $ \{(t,x): x-ct\in (-a_0, a_0) \} $.

Now let $a>a_0$ be given and assume by contradiction that there is $m>0$ and 
 a sequence $  t_n\to -\infty $ such that $ \underset{x\in(-a, a)}\sup\max \,(u,v)( t_n, x)\geq m$. Thanks to the Harnack inequality for parabolic systems, see \cite[Theorem 3.9]{Fol-Pol-09}, there is $C>0$ such that
$$
	\forall n\in\mathbb N,\quad  \inf_{x\in(-a, a)}\min\, (u,v)( t_n+1, x)\geq \frac 1C \sup_{x\in (-a, a)} \max\, (u+v)( t_n, x)\geq \frac mC.
$$
We now use our forward-in-time lower estimate, see Lemma \ref{lem:forward}, in $ (-a, a) $ and with $ \alpha:=\frac 12\min(\alpha_0, \frac mC) >0$ to get
$$
\forall n\in \mathbb N,\; \forall t> t_n+1,\; \forall x\in(-a, a),\; \vectorize{u(t,x)\\v(t,x)}\geq\alpha\vectorize{\varphi^a(x) \\ \psi^a(x)}.
$$
Since $t_n\to -\infty$ and $a>a_0$, the above implies
$$
    \underset{(t,x)\in\mathbb R\times(-a_0, a_0)}\inf\min \, (u,v)(t,x)\geq\alpha\inf_{x\in(-a_0, a_0)}(\varphi^a, \psi^a)(x)>0.
$$
This is a contradiction and the lemma is proved.
\end{proof}

\addcontentsline{toc}{section}{Appendix}

\appendix

\thereiam
\section{Topological theorems}\label{s:topo}

Let us first recall  the classical Krein-Rutman theorem. 

\begin{theo}[Krein-Rutman theorem]\label{theo:krein-rutman}
Let $E$ be a Banach space. Let $ C \subset E$ be a closed convex cone of vertex 0,
such that $ C\cap -C=\{0\}$ and $ Int\,C\neq\varnothing$.
Let $ T:E\to E $ be a linear compact operator  such that $
T(C\backslash\{0\})\subset Int\,C$.

Then, there exists $ u\in
Int\,C $ and $ \l>0 $ such that $ Tu=\l u$. Moreover,
if $ Tv=\mu v $ for some $ v\in C\backslash\{0\}$, then $
\mu=\l$. Finally, we have
\begin{equation*}
\l=\max\{|\mu|, \mu\in\sigma(T)\},
\end{equation*}
and the algebraic and geometric multiplicity of $ \l $ are
both equal to 1.
\end{theo}

We now quote some results on the structure of the solution set for nonlinear eigenvalue problems in a Banach space, more specifically when bifurcation occurs. 
For more details and proofs, we refer the reader to the works of Rabinowitz \cite{Rab-71-1, Rab-71}, Crandall and Rabinowitz \cite{Cran-Rab-1971}. See also earlier related results of Krasnosel'skii  \cite{Kra-64} and the book of Brown \cite{Brown}. 

\begin{theo}[Bifurcation from eigenvalues of odd multiplicity]\label{theo:krasnoleski-rabinowitz}
Let $ E $ be a Banach space. Let $ F:\mathbb R \times E\rightarrow
E $ be a (possibly nonlinear) compact operator such that
\begin{equation*}
 \forall \lambda\in\mathbb R,\, F(\lambda, 0)=0.
 \end{equation*}
Assume that $F$ is Fr\'echet differentiable near $ (\lambda, 0) $ with derivative $ \lambda T$. Let us define
\begin{equation*}
\mathcal S:=\overline{\{(\lambda, x)\in \mathbb R\times
E\backslash\{0\}:F(\lambda, x)=x\}}.
\end{equation*}
Let us assume that  $ \frac 1\mu\in\sigma(T) $ is of odd multiplicity.

Then there
exists a maximal connex compound $ \mathcal C_\mu \subset\mathcal
S $ such that $ (\mu, 0)\in{\mathcal C_\mu} $ and either
\begin{enumerate}
\item $ \mathcal C_\mu $ is not bounded in $ \mathbb R\times E$, or

 \item there exists $ \mu^*\neq\mu $ with $
\frac{1}{\mu^*}\in\sigma(T) $ and $ (\mu^*, 0)\in{\mathcal C_\mu}
$.
\end{enumerate}
\end{theo}

When the eigenvalue is simple, one can actually refine the above result as follows.

\begin{theo}[Bifurcation from simple eigenvalues]\label{theo:rabinowitz} Let the assumptions of Theorem \ref{theo:krasnoleski-rabinowitz} hold.
Assume further that $ \frac 1\mu\in\sigma(T) $ is simple. Let $ T^* $ be the dual
of $ T$, and $ l\in E' $ an eigenvector of $ T^* $ associated with $
\frac 1\mu $ with $ \Vert l\Vert=1 $ (recall that $ \frac 1\mu $
is of multiplicity 1 for both $ T $ and $ T^*$). Let us define
\begin{equation*}
K_{\xi, \eta}^+:=\{(\lambda, u)\in\mathbb R\times
E,|\lambda-\mu|<\xi, \langle l, u\rangle>\eta\Vert u\Vert\},
\quad K_{\xi, \eta}^-:=-K_{\xi, \eta}^+.
\end{equation*}

Then 
$\mathcal C_\mu\backslash\{(\mu, 0)\} $ contains two connex
compounds $ \mathcal C_\mu^+ $ and $ \mathcal C_\mu^- $ which
satisfy
\begin{equation*}
 \forall \nu\in\{+, -\},\forall \xi>0, \forall\eta\in(0, 1), \exists\zeta_0>0, \forall\zeta\in(0, \zeta_0),\,  (\mathcal C_\mu^\nu\cap B_\zeta)\subset K_{\xi, \eta}^\nu,
\end{equation*}
where $B_\zeta:=\{(\lambda, u)\in\mathbb R\times E, |\lambda-\mu|<\zeta,
\Vert u\Vert<\zeta\}$ is the ball of center $(\mu,0)$ and radius $\zeta$.
Moreover, both $ \mathcal C_\mu^+ $ and $ \mathcal C_\mu^- $
satisfies the alternative in Theorem
\ref{theo:krasnoleski-rabinowitz}.
\end{theo}

\section{A Bernstein-type interior gradient estimate for elliptic systems}\label{s:bernstein}

We present here some $L^\infty$ gradient estimates for regularizations of degenerate elliptic systems, which are uniform with respect to the regularization parameter $\kappa\geq 0$. The result below generalizes the result of Berestycki and Hamel \cite{Ber-Ham-05}, which is concerned with scalar equations.

\begin{lem}[Interior gradient estimates]\label{lem:bernsteinestimate}
Let $ \Omega $ be an open subset of $ \mathbb R^2$. Let $f,g:\Omega\times \R^2\to \R$ be two  $ C^1 $ functions with bounded derivatives. Let $ 0\leq \kappa\leq 1 $ and $ (u(y,x),v(y,x)) $ be a solution of the class $ C^3 $ of the system
\begin{equation}\label{eq:lembernsteinsystem}
\left\{\begin{array}{rcl}
-\kappa u_{yy}-u_{xx}+u_y&=&f(y, x, u, v) \quad \text{ in } \Omega,\\
-\kappa v_{yy}-v_{xx}+v_y&=&g(y, x, u, v) \quad \text{ in } \Omega.
\end{array}\right.
\end{equation}

Then, for all $ (y, x) \in\Omega$, 
\begin{equation*}
|u_x(y, x)|^2+|v_x(y, x)|^2+\kappa |u_y(y,x)|^2 + \kappa|v_y(y,x)|^2\leq C\left(1+\frac{1}{(dist((y,x), \partial\Omega))^2}\right)
\end{equation*}
where
$$
C=C( \Vert u\Vert_{L^\infty(B)}+\Vert v\Vert_{L^\infty(B)}, osc_{B}u, osc_{B}v, \Vert f\Vert_{C^{0, 1}(B\times[\underline u, \overline u]\times [\underline v, \overline v])},
\Vert g\Vert_{C^{0,1}(B\times [\underline u, \overline u]\times[\underline v, \overline v])}),
$$
with $ B$ the ball of center $(y, x)$ and radius $\frac{dist((y, x),\partial\Omega)}{2}$ in $\R^2$, 
$ \underline u:=\inf_{B} u $, $ \overline u := \sup _{B} u $,
$ \underline v:=\inf _{B} v $, $ \overline v := \sup _{B} v $. 
In particular, this estimate is independent on the regularization parameter $ 0\leq \kappa\leq 1$.  
\end{lem}

\begin{proof}
Let $ h $ be the smooth function defined on $ \mathbb R $ by 
$$
h(z):=\left\{\begin{array}{lcl}
\exp\left(\frac{z^2}{z^2-1}\right) &\qquad& |z|< 1 \\
0&\qquad& |z|\geq 1.
\end{array}\right.
$$
Let us then define 
 $ C_0:=\max(\Vert h\Vert_{L^\infty}, \Vert h'\Vert_{L^\infty}, \Vert h''\Vert_{L^\infty})$ and 
$ \zeta(Y, X):=h\left(\frac{\sqrt{Y^2+X^2}}{2}\right) $.

Let $ (y_0, x_0)\in\Omega$ be a given point, $ d_0:=dist((y_0, x_0), \partial\Omega)$,  $ d:=\min\left( \frac{d_0}{2}, 1\right) $, $B_0$ the ball of center $(y_0,x_0)$ and radius $d$. Let
$ \chi $ be the function defined by 
\begin{equation*}
 \forall (y, x)\in\mathbb R^2, \quad \chi(y, x):=\zeta\left(\frac{y-y_0}{d}, \frac{x-x_0}{d}\right).
 \end{equation*}
Finally, let $ P^u$ and $ P^v $ be defined in $ \Omega $ by 
\begin{eqnarray*}
P^u(y, x)&:=&\chi^2(y, x)(u_x^2(y,x) + \kappa u_y^2(y,x))+\lambda u^2(y,x)+\rho e^{x-x_0} \\
P^v(y, x)&:=&\chi^2(y, x)(v_x^2(y,x) + \kappa v_y^2(y,x))+\lambda v^2(y,x)+\rho e^{x-x_0},
\end{eqnarray*}
where $ \lambda >0$ and $ \rho >0$ are constants to be fixed later. Our goal is to apply the maximum principle to the function $ P:=P^u+P^v $ for convenient values of $ \lambda $ and $ \rho$. We present below the computations on $ P^u $ only and reflect them on $ P^v$. 

We first compute the partial derivatives of $ P^u $ and get
\begin{eqnarray*}
P^u_y&=&2\chi_y\chi u_x^2+2\chi^2u_{xy}u_x+2\kappa(\chi_y\chi u_y^2+\chi^2u_{yy}u_y)+2\lambda u_yu \\
P^u_{yy}&=&2(\chi_{yy}\chi + \chi_y^2)u_x^2 + 8\chi_y\chi u_{xy}u_x + 2\chi^2(u_{xyy}u_x+u_{xy}^2) \\
\nonumber & & + \kappa[2(\chi_{yy}\chi + \chi_y^2)u_y^2 + 8\chi_y\chi u_{yy}u_y + 2\chi^2(u_{yyy}u_y+u_{yy}^2)]\\
\nonumber & & + 2\lambda(u_{yy}u + u_y^2) \\
P^u_{xx}&=&2(\chi_{xx}\chi + \chi_x^2)u_x^2 + 8\chi_x\chi u_{xx}u_x + 2\chi^2(u_{xxx}u_x+u_{xx}^2) \\
\nonumber & & \kappa [2(\chi_{xx}\chi + \chi_x^2)u_y^2 + 8\chi_x\chi u_{xy}u_y + 2\chi^2(u_{yxx}u_y+u_{yx}^2)] \\
\nonumber & & + 2\lambda(u_{xx}u + u_x^2) + \rho e^{x-x_0}.
\end{eqnarray*}
Let $ M:= \partial _y- \kappa\partial_{yy}-\partial_{xx}$. Then we have
\begin{equation*}
\begin{array}{rcl} 
	MP^u & = & 2\left[\chi_y\chi - \kappa(\chi_{yy}\chi+\chi_y^2) - (\chi_{xx}\chi+\chi_x^2)\right]u_x^2 \\
	     &   & + 2\kappa \left[\chi_y\chi -\kappa (\chi_{yy}\chi+\chi_y^2)-(\chi_{xx}\chi+\chi_x^2)\right]u_y^2 \\
	     &   & + 2\chi^2\left[u_{xy}-\kappa u_{xyy}- u_{xxx}\right]u_x \\
	     &   & + 2\kappa \chi^2\left[u_{yy}-\kappa u_{yyy}-u_{yxx}\right]u_y \\
	     &   & - 2 \left[\kappa (\chi^2u_{xy}^2 + 4 \chi_y\chi u_{xy}u_x ) + (\chi^2u_{xx}^2 + 4 \chi_x\chi u_{xx} u_x)\right] \\
	     &   & - 2\kappa[\kappa(4\chi_y\chi u_yu_{yy} + \chi^2u_{yy}^2) + (4\chi_x \chi u_y  u_{xy} + \chi^2 u_{xy}^2)] \\
	     &   & + 2 \lambda \left[ (u_y-\kappa u_{yy}-u_{xx}) u -\kappa u_y^2 - u_x^2 \right] \\
	     &   & - \rho e^{x-x_0}.
\end{array}
\end{equation*}
We now reformulate some of the lines of the above equality, starting with lines three and four. We differentiate the first equation of system \eqref{eq:lembernsteinsystem} with respect to $ x $ to obtain
\begin{equation*}
\begin{array}{rcl}
 2\chi^2\left[u_{xy}-\kappa u_{xyy} - u_{xxx}\right]u_x & = & 2\chi^2 (f_x + u_xf_u + v_xf_v)u_x \\
 															 & \leq & \chi^2(u_x^2 +f_x^2) + 2\chi^2u_x^2|f_u| + \chi^2(u_x^2+v_x^2)|f_v|,
\end{array}
\end{equation*}
and then with respect to $ y $ to get
\begin{equation*}
\begin{array}{rcl}
 2\chi^2\left[u_{yy}-\kappa u_{yyy} - u_{yxx}\right]u_y & = & 2\chi^2 (f_y + u_yf_u + v_yf_v)u_y \\
 															 & \leq & \chi^2(u_y^2 +f_y^2) + 2\chi^2u_y^2|f_u| + \chi^2(u_y^2+v_y^2)|f_v|.
\end{array}
\end{equation*}
As far as lines five and six are concerned, we use the factorizations
\begin{eqnarray*}
\chi^2u_{xy}^2 + 4 \chi_y\chi u_{xy}u_x  &=& (\chi u_{xy} + 2\chi_y u_x)^2 - 4\chi_y^2u_x^2\\
\chi^2u_{xx}^2 + 4 \chi_x\chi u_{xx} u_x &=& (\chi u_{xx} + 2 \chi_x u_x)^2 - 4\chi_x^2u_x^2 \\ 
\chi^2u_{yy}^2 + 4 \chi_y\chi u_{yy} u_y &=& (\chi u_{yy} + 2 \chi_y u_y)^2 - 4\chi_y^2u_y^2 \\
\chi^2u_{xy}^2 + 4 \chi_x\chi u_{xy} u_y &=& (\chi u_{xy} + 2 \chi_x u_y)^2 - 4\chi_x^2u_y^2.
\end{eqnarray*}
For line seven, we use the first equation in  \eqref{eq:lembernsteinsystem} 
to write $(u_y-\kappa u_{yy}-u_{xx}) u =  f u$.
As a result, we collect
\begin{equation*}
\begin{array}{rcl}
MP^u & \leq & 2\left[\chi_y\chi - \kappa\chi_{yy}\chi - \chi_{xx}\chi + 3\chi_x^2+3\kappa \chi_y^2+ \chi^2\left(|f_u|+\frac{1+|f_v|}{2}\right) - \lambda\right](u_x^2+\kappa u_y^2) \\
    &      & +2\lambda fu +\chi^2(v_x^2+\kappa v_y^2)|f_v| + \chi^2(f_x^2 + \kappa f_y^2) -\rho e^{x-x_0},
\end{array}
\end{equation*}
and, similarly,
\begin{equation*}
\begin{array}{rcl}
MP^v & \leq & 2\left[\chi_y\chi - \kappa\chi_{yy}\chi - \chi_{xx}\chi + 3\chi_x^2+3\kappa \chi_y^2+ \chi^2\left(|g_v|+\frac{1+|g_u|}{2}\right) - \lambda\right](v_x^2+\kappa v_y^2) \\
     &      & +2\lambda gv +\chi^2(u_x^2+\kappa u_y^2)|g_u| +\chi^2(g_x^2+\kappa g_y^2) -\rho e^{x-x_0}.
     \end{array}
\end{equation*}
Notice that $\vert \chi\vert \leq C_0$, $\vert \chi _x\vert, \vert \chi _y\vert \leq \frac{C_0}{d}$, $\vert \chi _{xx}\vert, \vert \chi _{yy}\vert \leq \frac{C_0}{d^{2}}$ and recall that $\kappa, d \leq 1$. Hence, putting everything together, we arrive at
\begin{equation*}
\begin{array}{rcl}
MP & \leq & \left( 20 \frac{C_0^2}{d^2} + 4C_0^2(\Vert f\Vert_{C^{0,1}} + \Vert g\Vert_{C^{0,1}}) + C_0^2 - \lambda\right)(u_x^2+v_x^2+\kappa u_y^2+\kappa v_y^2) \\
 &  & + 2\lambda(\Vert f\Vert_{L^\infty}+\Vert g\Vert_{L^\infty})(\Vert u\Vert_{L^\infty}+\Vert v\Vert_{L^\infty}) + 2C_0^2(\Vert f\Vert_{C^{0, 1}}^2 + \Vert g\Vert_{C^{0, 1}}^2) - 2 \rho e^{x-x_0}.
\end{array}
\end{equation*}

It is now time to specify
\begin{equation*}
\left\{\begin{array}{rcl}
\lambda & = & 20\frac{C_0^2}{d^2} + 4C_0^2(\Vert f\Vert_{C^{0,1}} + \Vert g\Vert_{C^{0,1}}) + C_0^2 > 0 \\
\rho & = & \frac e2\left[2\lambda(\Vert f\Vert_{L^\infty}+\Vert g\Vert_{L^\infty})(\Vert u\Vert_{L^\infty}+\Vert v\Vert_{L^\infty}) +2C_0^2(\Vert f\Vert_{C^{0, 1}}^2 + \Vert g\Vert_{C^{0, 1}}^2) + 1\right] > 0.
\end{array}\right.
\end{equation*}
As a result we have $MP(y,x)<0$ for all $(y,x)\in B_0$ (since then $x-x_0\geq -1$). The maximum principle then implies
\begin{equation*}
P(y_0, x_0)\leq\underset{(y, x)\in\partial B_0}{\max}P(y,x).
\end{equation*}
Since $\chi(y_0,x_0)=1$ and $\chi(y,x)=0$ when $(y,x)\in \partial B_0$, the above inequality implies
\begin{eqnarray*}
(u_x^2+v_x^2 + \kappa u_y^2+\kappa v_y^2)(y_0,x_0)&\leq& \lambda(\Vert u\Vert_{L^\infty}^2+\Vert v\Vert_{L^\infty}^2) -\lambda (u^2+v^2)(y_0, x_0) + 2\rho e\\
 & \leq & 2\lambda (\Vert u\Vert_{L^\infty} osc_{B_0}(u) + \Vert v\Vert_{L^\infty} osc_{B_0}(v)) + 2\rho e \\
\nonumber & \leq & K\{(\Vert u\Vert_{L^\infty}+\Vert v\Vert_{L^\infty})(osc_{B_0}(u)+osc_{B_0}(v)\\
&&+\Vert f\Vert_{C^{0, 1}} + \Vert g\Vert_{C^{0, 1}})  +  \Vert f\Vert_{C^{0, 1}}^2 + \Vert g\Vert_{C^{0, 1}}^2+ 1\}\left(1+\frac{1}{d^2}\right)
\end{eqnarray*}
using the expressions of $\lambda$ and $\rho$ above, for a universal positive constant $ K>0$ and where the $C^{0,1}$ norms of $ f$, $ g $  are taken on $ B_0\times [\inf_{B_0} u, \sup _{B_0} u]\times [\inf_{B_0}v, \sup_{B_0}v]$. This proves the lemma.
\end{proof}

\section{Dirichlet and periodic principal  eigenvalues}\label{s:diri-perio}

We prove here that the principal eigenvalue with Dirichlet boundary conditions in a ball converges to the principal eigenvalue with periodic boundary conditions, when the radius tends to $ +\infty$.

\begin{lem}[Dirichlet and periodic principal eigenvalues]\label{lem:annexeC}
Let $A\in L^\infty(\mathbb R; \mathcal S_2(\mathbb R)) $ be a symmetric cooperative matrix field  that is periodic with period $ 
L>0$. Let $\l$ be the principal eigenvalue of the operator $-\partial_{xx}-A(x)$ with  periodic boundary conditions, that is
\begin{equation}\label{bidule-chouette}
-\left(\begin{matrix} \varphi \\ \psi\end{matrix}\right)''-A(x)\left(\begin{matrix} \varphi \\ \psi \end{matrix}\right) = \l\left(\begin{matrix} \varphi \\ \psi \end{matrix}\right),
\end{equation}
with $ \varphi, \psi \in H^1_{per} $ and $ \varphi>0, \psi>0$. For $R>0$, let  $ \lambda_1^{R}$ be the principal eigenvalue of the operator $-\partial _{xx}-A(x)$
with Dirichlet boundary conditions on $(-R,R)$, that is
\begin{equation}\label{pb:appendixeigenDirichlet}
-\left(\begin{matrix} \varphi ^R \\ \psi^R\end{matrix}\right)''-A(x)\left(\begin{matrix} \varphi ^R\\ \psi ^R \end{matrix}\right) = \lambda_1^{R}\left(\begin{matrix} \varphi ^R \\ \psi ^R\end{matrix}\right),
\end{equation}
with $ \varphi^R, \psi^R \in H^1_{0}(-R, R) $ and $ \varphi^R>0, \psi^R>0$. Then,
there exists $C>0$ depending only on $A$ such that, for all $R>0$, 
\begin{equation*}
\l\leq\lambda_1^{R}\leq \l+\frac{C}{R}.
\end{equation*}
\end{lem}

\begin{proof} Without loss of generality we assume $L=1$. Inequality $\l\leq \lambda _1 ^R$ is very classical, see \cite[Proposition
4.2]{ber-ham-ros} or \cite[Proposition 3.3]{Alf-Ber-Rao} for instance, and we omit the details. Also, the same classical argument yields that $R\mapsto \lambda _1 ^R$ is nonincreasing so it is enough to prove $\lambda _1^R \leq \l +\frac C R$ when $R=2,3,...$.

 We consider a smooth auxiliary function $\eta:\R\to\R$ satisfying 
\begin{equation*}
\eta \equiv 1 \text{ on } (-\infty,0],\; 0<\eta<1 \text{ on } (0,1), \; \eta\equiv 0 \text{ on } [1,\infty).
\end{equation*}
 Since the operator in \eqref{pb:appendixeigenDirichlet} is self-adjoint in the domain $ (-R, R) $, the principal eigenvalue $\lambda _1  ^{R}$ is given by the Rayleigh quotient
\begin{equation*}\label{eq:infrayleigh}
\lambda_1^{R}=\underset{\Psi\in \mathbf H^1_0(-R,R), \Psi\neq 0}{\inf}Q(\Psi,\Psi),\quad  Q(\Psi, \Psi):=\frac{\int_{-R}^R ( \,^t\Psi_x\Psi_x - \,^t\Psi A(x)\Psi)\mathrm dx}{\int_{-R}^R\,^t\Psi\Psi\mathrm dx}.
\end{equation*}
In particular we have $\lambda _1 ^{R}\leq Q(\Theta,\Theta)$, with $\Theta$
 the $\mathbf H^1_0(-R,R)$ test function defined by 
\begin{equation*}
\Theta (x):=\eta(-R+1-x)\eta(-R+1+x)\Phi(x),\quad \Phi(x):=\left(\begin{matrix} \varphi(x) \\ \psi(x) \end{matrix}\right),
\end{equation*}
where $\varphi, \psi$ are as in \eqref{bidule-chouette}, with the normalization $\int_0 ^{1}\,^t\Phi \Phi \mathrm dx=1$.  We then have 
$Q(\Theta, \Theta)=Q^1(\Theta) + Q^2(\Theta)$, where
$$
Q^1(\Theta):=\frac{\int _{\vert x\vert \leq R-1}(\,^t\Theta_{x}\Theta_{ x} - \,^t\Theta A(x)\Theta )\mathrm dx}{\int_{-R}^R\,^t\Theta\Theta\mathrm dx},\quad
Q^2(\Theta):=\frac{\int _{R-1\leq \vert x\vert \leq R}(\,^t\Theta_{ x}\Theta_{ x} - \,^t\Theta A(x)\Theta)\mathrm dx}{\int_{-R}^R\,^t\Theta\Theta \mathrm dx}.
$$
We write
$$
Q^1(\Theta)=\frac{\int _{\vert x\vert \leq R-1}(\,^t\Theta_{x}\Theta_{ x} - \,^t\Theta A(x)\Theta )\mathrm dx}{\int_{\vert x\vert \leq R-1}\,^t\Theta \Theta\mathrm dx} \frac{\int_{-(R-1)}^{R-1}\,^t\Theta\Theta \mathrm dx}{\int_{-R}^{R}\,^t\Theta \Theta \mathrm dx}=\lambda _1 \frac{\int_{-(R-1)}^{R-1}\,^t\Theta\Theta \mathrm dx}{\int_{-R}^{R}\,^t\Theta \Theta \mathrm dx},
$$
thanks to $\Theta\equiv \Phi \equiv \left(\begin{matrix} \varphi \\ \psi\end{matrix}\right)$ on $(-(R-1),R-1)$ and the 1-periodicity of $\varphi$, $\psi$ (recall that $R-1$ is an integer). As a result 
$$
\vert Q ^1(\Theta)-\lambda _1\vert =\vert \lambda _1\vert \frac{\int_{R-1<\vert x\vert <R}\,^t\Theta\Theta \mathrm dx}{\int_{-R}^{R}\,^t\Theta \Theta \mathrm dx}\leq \vert \lambda _1\vert \frac{\int_{R-1<\vert x\vert <R}\,^t\Phi\Phi \mathrm dx}{\int_{-(R-1)}^{R-1}\,^t\Phi \Phi \mathrm dx}
=\vert \lambda _1\vert \frac{1}{R-1},
$$
since $0\leq \eta \leq 1$. On the other hand one can see that, for a constant $C_2>0$ depending only on $\Vert \eta '\Vert _{L^{\infty}(\R)}$ and $ \Vert A\Vert_{L^\infty(\R; \mathcal S _2(\R))}$, 
\begin{eqnarray*}
\left \vert \int _{R-1< \vert x\vert < R}(\,^t\Theta_{ x}\Theta_{ x} - \,^t\Theta A(x)\Theta)\mathrm dx\right \vert &\leq& C_2 \int_{R-1<\vert x\vert <R}\,(^t\Phi\Phi+^t\Phi_x\Phi_x) \mathrm dx
\\
&=&2 C_2 \int_{0<\vert x\vert <1}\,(^t\Phi\Phi+^t\Phi_x\Phi_x) \mathrm dx=:C_2'
\end{eqnarray*}
so that
$$
\vert Q^{2}(\Theta)\vert \leq \frac{C_2'}{\int_{-R}^R\,^t\Theta\Theta\mathrm dx}\leq \frac{C_2'}{\int_{-(R-1)}^{R-1}\,^t\Phi\Phi\mathrm dx}= \frac{C_2'}{(2R-2)\int _0 ^{1}\,^t\Phi\Phi\mathrm dx}=\frac{C_2'}{(2R-2)}.
$$
This concludes the proof of the lemma.
\end{proof}

\bibliographystyle{siam}    
\bibliography{Biblio_q}

\end{document}